\documentclass{amsart}
\usepackage[dvips]{color}
\usepackage{graphicx}
\usepackage{epsfig}
\usepackage{tikz}
\usepackage{enumerate}
\usepackage{color}

\usepackage{latexsym}
\usepackage{amssymb}
\usepackage{amsmath}
\usepackage{amsthm}
\usepackage{amscd}

\usepackage[utf8]{inputenc}
\usepackage{xcolor}
\usepackage{tikz-cd}
\usepackage{siunitx}
\usetikzlibrary{matrix, positioning, calc}
\usepackage{hyperref}
\usepackage[shortlabels]{enumitem}

\title[Ergodic theory of real Rel]{On the ergodic theory of the real Rel foliation} 
\author{Jon Chaika}
\address{University of Utah {\tt chaika@math.utah.edu }}

\author{Barak Weiss}
\address{Tel Aviv University 
{\tt barakw@tauex.tau.ac.il}}

\font\sn = cmssi8 scaled \magstep0

\long\def\combarak#1{\ifdraft{\color{red}\sn #1 }\else\ignorespaces\fi}

\numberwithin{equation}{section}

\newcommand{\C}{\mathbb{C}}

\newcommand{\q}{\mathbf{q}}
\newcommand{\Rel}{\mathrm{Rel}}

\newcommand\ec{/\!\!/}

\newcommand\dev{\mathrm{dev}}
\newcommand\hol{\mathrm{hol}}
\newif\ifdraft\drafttrue
\draftfalse
\newcommand\name[1]{\label{#1}{\ifdraft{\sn [#1]}\else\ignorespaces\fi}}

\newcommand\eq[2]{{\ifdraft{\ \tt [#1]}\else\ignorespaces\fi}\begin{equation}\label{#1}{#2}\end{equation}}
\newcommand {\equ}[1]{\eqref{#1}}

\newcommand{\HH}{{\mathcal{H}}}

\newcommand{\Q}{{\mathbb {Q}}}

\newcommand{\R}{{\mathbb{R}}}

\newcommand{\Res}{{\mathrm{Res}}}

\newcommand{\Prob}{{\mathrm{Prob}}}

\newcommand{\N}{{\mathbb{N}}}

\newcommand{\GL}{\operatorname{GL}}

\newcommand{\Mod}{\operatorname{Mod}}

\newcommand{\SL}{\operatorname{SL}}

\newcommand {\ignore}[1]  {}

\newcommand{\spa}{{\rm span}}

\newcommand{\dist}{{\rm dist}}

\newcommand{\diam}{{\rm diam}}

\newcommand{\bq}{{\mathbf{q}}}

\newcommand{\LL}{{\mathcal L}}
\newcommand{\cL}{{\LL}}

\newcommand{\df}{{\, \stackrel{\mathrm{def}}{=}\, }}

\newcommand{\hs}{\kern 0.8pt}

\newcommand{\supp}{{\rm supp}}

\newcommand{\rel}{{\mathrm{Rel}}}

\newcommand{\sm}{\smallsetminus}

\newcommand{\vre}{\varepsilon}

\newtheorem{thm}{Theorem}[section]

\newtheorem{lem}[thm]{Lemma}

\newtheorem{prop}[thm]{Proposition}

\newtheorem{cor}[thm]{Corollary}

\newtheorem{remark}[thm]{Remark}

\newcommand{\cH}{\mathcal{H}}

\begin{document}
\maketitle

\begin{abstract}
Let $\HH$ be a stratum of translation surfaces with at least two
singularities, let $m_{\HH}$ denote the Masur-Veech measure on $\HH$,
and let $Z_0$ be a flow on $(\HH, m_{\HH})$ obtained by integrating
a Rel vector field. We prove that $Z_0$ is mixing of all orders, and
in particular is ergodic. We also characterize the ergodicity of flows
defined by Rel vector fields, for more general spaces $(\LL, m_{\LL})$,
where $\LL \subset \HH$ is an orbit-closure for the action of $G =
\SL_2(\R)$ (i.e., an affine invariant subvariety) and $m_{\LL}$ is the
natural measure. These results are conditional on a forthcoming measure
classification result of Brown, Eskin, Filip and Rodriguez-Hertz. We
also prove that the entropy of $Z_0$ with respect to any of the
measures $m_{\cL}$ is zero. 
\end{abstract}

\section{Introduction}
Let $\HH$ be a stratum of area-one translation surfaces and let $G \df
\SL_2(\R)$. There is a $G$-invariant finite measure $m_{\HH}$ on $\HH$
known as the {\em Masur-Veech measure}, and the dynamics of the
$G$-action on $(\HH, m_{\HH})$ have been intensively 
studied in recent years and are intimately connected to many problems
in geometry and ergodic theory, see e.g. \cite{MT, zorich
  survey}. Suppose that surfaces in $\HH$ have $k$ singularities, 
where $k\geq 2$. Then there is a $k-1$-dimensional foliation of $\HH$,
known as the {\em real Rel foliation}. A precise definition of the
foliation and some of its properties will be given below in \S \ref{subsec:
  rel foliation}. Loosely
speaking, two surfaces are in the same real Rel leaf if one can be
obtained from the other by a surgery in which singular points are
moved with respect to each other in the horizontal direction, without
otherwise changing the geometry of the surface. A natural question,
which we address here, is the ergodic properties of this foliation.

As we review in \S \ref{subsec: rel foliation}, by labeling the
singularities and removing a set of
leaves of measure zero, we 
can think of the real Rel leaves as being the orbits of an action of a
group $Z$ on $\HH$, where $Z \cong \R^{k-1}$, and the restriction of this
action to any one-dimensional subgroup of $Z$ defines a flow. Our first main result is
the following.

  \begin{thm}\label{thm: main}
Let $\HH$ be a connected component of a stratum $\HH(a_1, \ldots,
a_k)$ with all $a_i>0$ (i.e., no marked points). Let $m_\HH$ be
the Masur-Veech measure on $\HH$, let $Z \cong \R^{k-1}$ be the
corresponding action given by translation along the leaves of the real Rel
foliation, and let $Z_0 \subset Z$ be any 
one-dimensional connected subgroup of $Z$. Then the $Z_0$-flow on
$(\HH, m_{\HH})$ is 
mixing of all orders (and in 
particular, ergodic). 
\end{thm}

The definition of mixing of all orders is given in \S \ref{subsec:
  mixing of all orders}. For purposes of this introduction it is
enough to note that it implies ergodicity of any nontrivial element. 
Note that when $\HH$ has marked points, there will be subgroups $Z_0$
which only move the marked points on surfaces without otherwise
changing the geometry, and the conclusion of Theorem \ref{thm: main}
will not hold. This is the only obstruction to generalizing our
results to strata with marked points, see Theorem \ref{thm: main 1}.

The proof of Theorem \ref{thm: main}, as well as most of the other
  results of this paper, relies crucially on measure-rigidity results
  of Eskin, Mirzakhani and Mohammadi
  \cite{EM, EMM}, and further forthcoming work extending these
  results, which we will describe in \S \ref{sec: magic wand upgraded}. 

Theorem \ref{thm: main} improves on the results of several authors. In
those results, ergodicity for the {\em full Rel foliation} was studied. The
full Rel foliation (also referred to as the `kernel foliation',
`isoperiodic foliation', or
`absolute period foliation') will also be defined in \S \ref{subsec:
  rel foliation}. Its leaves are of dimension $2(k-1)$, that is, twice
the dimension of the real Rel leaves. Loosely speaking, two surfaces
are in the same leaf for this foliation if one can be obtained from
the other by moving the singularities (without otherwise affecting the
geometry of the surface). That is, we relax the hypothesis that points
can only be moved horizontally. The first ergodicity results for the
full Rel foliation were obtained by McMullen \cite{McM}, who proved ergodicity
in the two strata $\HH(1,1)$ and
$\HH(1,1,1,1)$. Subsequently, Calsamiglia, Deroin and Francaviglia
\cite{CDF} proved ergodicity in all principal strata, and Hamenst\"adt
\cite{Ham} reproved their result by a simpler argument.
Recently, Winsor \cite{Winsor} 
proved ergodicity for most of the additional strata, and in
\cite{Winsor 2}, showed that  there are dense orbits for the
$Z_0$-flow, for any $Z_0$ as in Theorem \ref{thm:
  main}. Note that ergodicity
for a foliation is implied by ergodicity for any of its subfoliations,
and that ergodicity implies the existence of dense leaves,
and thus Theorem \ref{thm: main} generalizes all of these
results. Also note that full Rel is
a foliation which is not given by a group action, and the notions of 
mixing and multiple mixing do not make sense in this case.

The papers \cite{McM, CDF} go beyond
ergodicity and obtain classifications of full Rel closed leaves and
leaf-closures in their respective 
settings. We suspect that there is not a reasonable classification of
real Rel leaf-closures, indeed it is already known (see \cite{HW}) that there are real
Rel trajectories that leave every compact set never to return. 

The strata $\HH$ support other interesting measures for which similar
questions could be asked. Namely, by work of Eskin, Mirzakhani and
Mohammadi \cite{EM, EMM}, for any $q \in \HH$, the
orbit-closure $\cL \df \overline{Gq}$ is the support of  a unique
smooth $G$-invariant measure which we denote by $m_{\cL}$. Let
$Z_{\cL}$ be the subgroup of $Z$ leaving $\cL$ invariant. Then
$Z_{\cL}$ also preserves $m_{\cL}$ and for many choices of $\cL$, we
have $\dim Z_\cL>0$. In these cases, for any closed connected $Z_1 \subset Z_{\cL}$,
there is a {\em complexification} 
$\mathfrak{R}_{1}$, which gives a foliation of $\cL$ whose leaves
$\mathfrak{R}_{1}(q)$ 
have dimension $2 \dim Z_1$ (see \S \ref{subsec: rel 
  foliation}). The leaves $\mathfrak{R}_1(q)$ have a natural
translation structure, and this induces a natural locally finite
translation-invariant 
measure on each leaf. With this terminology we can now state the main
result of this paper:



\begin{thm} \label{thm: general} Let $\cL$ be a $G$-orbit closure, and let
  $m_{\mathcal{L}}, \, Z_{\cL}, \mathfrak{R}_{\cL}$ be as above, where
  $\dim Z_{\cL}>0$. Let $z_0$ be a nontrivial element of $Z_{\cL}$ and
  let $Z_0 = \spa_{\R} (z_0).$ Then either  
\begin{enumerate}
\item \label{item: one}
The action of  $Z_0$ on  $\left (\cL, m_{\mathcal{L}} \right)$ is mixing of all orders (and in
  particular, $z_0$ acts ergodically); or 
\item \label{item: two}
  there is an intermediate closed connected subgroup $Z_1$ so that $Z_0\subset
  Z_1\subset Z_{\cL}$, and the 
  complexification $\mathfrak{R}_1$ of $Z_1$ satisfies 
\begin{itemize}
\item for every $q \in \cL$, the leaf $\mathfrak{R}_1(q)$ is closed, 
  and  
\item for $m_{\mathcal{L}}$-a.e. $q$, $\mathfrak{R}_1(q)$ is of finite
  volume with respect to its translation-invariant 
  measure,  and
  $\overline{Z_0q}=\mathfrak{R}_1(q)$. 
\end{itemize}
\end{enumerate}
\end{thm}

Thus, in order to establish ergodicity of real Rel subfoliations on
$G$-orbit-closures, it is enough to rule out Case \eqref{item:
  two}. We will prove Proposition \ref{prop: criterion}, which provides
a simple way to achieve this, using cylinder circumferences of
surfaces in $\cL$. Theorems \ref{thm: main} and \ref{thm: main 1} are
deduced from Theorem \ref{thm: general} using Proposition \ref{prop:
  criterion}. 

The following statement is an immediate consequence of Theorem
\ref{thm: general}. 
\begin{cor}\label{cor: dense leaf enough}
Let $\cL$ be a $G$-orbit-closure, let $m_{\cL}, Z_{\cL}$ be as
above, and let $Z_1 \subset Z_{\cL}$ be one-dimensional. Suppose that the foliation
induced by the 
complexification $\mathfrak{R}_1$ has a dense leaf. Then the 
$Z_1$-flow on  $(\cL, 
m_{\cL})$ is mixing of all orders (and in particular, ergodic).
\end{cor}

The density of certain leaves of the full Rel foliation in
$G$-orbit-closures of rank one  was
obtained by Ygouf in \cite{Ygouf}. Using these results we obtain 
ergodicity of one-dimensional subgroups of the real Rel foliation in
many cases. For instance, using \cite[Thm. A \& Prop. 5.1]{Ygouf} we have:

\begin{cor}
The real Rel foliation is mixing of all orders (and in particular,
ergodic) in any eigenform locus in $\HH(1,1)$ with a non-square
discriminant. 
\end{cor}

Recall that in \cite{Winsor 2} Winsor proved the existence of dense
real Rel leaves, and dense leaves for one-dimensional flows $Z_0$, in all
strata. Using these results in conjunction with Corollary \ref{cor:
  dense leaf enough}, one can 
obtain an alternative proof of Theorem \ref{thm: main} that avoids the use of
Proposition \ref{prop: criterion}.

We also consider the entropy of real-Rel flows, and show the
following:
\begin{thm}\label{thm: rel entropy}
  Let $\cL, m_{\cL}, Z_{\cL}, z_0$ be as in
  the statement of Theorem \ref{thm: general}. Then the entropy of the action of
  $\Rel_{z_0}$ on the measure space $(\cL, m_{\cL})$ is zero. 
\end{thm}

Using the geodesic flow one easily shows that $\Rel_{z_0}$ is
conjugate to $\Rel_{tz_0}$ for any $t >0$, and from this it follows that
the entropy is either zero or infinite. However, the Rel flow is
not continuous, and we could not find a
simple way to rule out infinite entropy. 
Our proof gives a more general result --- see Theorem \ref{thm: rel entropy general}. However, the
argument fails for $Z_0$-invariant measures for 
which the backward time geodesic flow diverges almost surely, and thus
we do not settle the question of whether the topological
entropy of real Rel flows is zero.

\subsection{Outline}
In \S\ref{sec: prelims} we give background material on translation surfaces, their
moduli spaces, and the Rel foliation. In \S \ref{sec: ergodic theory} we use standard
facts about joinings to build a measure $\theta$ on the product of two strata
(see \eqref{eq: we have}), depending on a real Rel flow $Z_0$, such that
if $\theta$ is the 
product measure, then $Z_0$ is ergodic. In \S \ref{subsec: mixing of all
  orders} we discuss a technique of Mozes 
that makes it possible to upgrade ergodicity
to mixing of all orders. In \S \ref{sec: Mautner} we show that $\theta$ is
ergodic for the diagonal action of the upper triangular group $P
\subset G$ on the product of two strata. In \S \ref{sec: magic
  wand upgraded} we state 
a far-reaching measure rigidity result of Brown, Eskin, Filip and
Rodriguez-Hertz for $P$-ergodic measures on products of two strata. In
\S \ref{sec: putting together} we use this measure rigidity result, as well as prior
results for the action on one stratum due to Wright, in order to
characterize the situations in which $\theta$ is not a product
measure, thus proving
Theorem \ref{thm: general}. Proposition \ref{prop: criterion} is proved in \S \ref{sec:
    topological condition}, and we check its conditions to deduce
  Theorems \ref{thm: main} and \ref{thm: main 1} in \S \ref{sec:
    checking condition}. We prove
  Theorem \ref{thm: 
    rel entropy} in \S \ref{sec: zero entropy}.  

  \subsection{Acknowledgements}
We are very grateful to Alex Eskin for crucial contributions to this
project. We also thank Simion Filip, Curt McMullen and Alex Wright for useful
discussions, and
acknowledge the support of ISF grant 2919/19, BSF grant 2016256,
NSFC-ISF grant 3739/21, a Warnock chair, a Simons Fellowship, 
and NSF grants DMS-1452762 and DMS-2055354.

\section{Preliminaries about translation surfaces}\label{sec: prelims}
\subsection{Strata, period coordinates} \label{subsec: strata}
In this section we collect
standard facts about translation surfaces, and fix our notation. For
more details, we refer to reader to \cite{zorich survey, Wright
  survey, eigenform}. Below we briefly summarize the treatment in
\cite[\S 2]{eigenform}.

Let $S$ be a compact oriented surface of genus $g$, $\Sigma = \{\xi_1,
\ldots, \xi_k \} \subset S$ a finite set, $a_1, \ldots, a_k$
non-negative integers with $\sum a_i = 2g-2$, and $\HH = \HH(a_1,
\ldots, a_k)$ the corresponding stratum of unit-area translation surfaces. We
let $\HH_{\mathrm{m}}=\HH_{\mathrm{m}}(a_1, \ldots, a_k)$ denote the
stratum of unit-area marked translation surfaces and $\pi : \HH_{\mathrm{m}} \to
\HH$ the forgetful mapping. Our convention is that singular points
are labeled, or equivalently, $\HH = \HH_{\mathrm{m}} / \Mod(S,
\Sigma)$, where $\Mod(S, \Sigma)$ is the group of isotopy classes of
orientation-preserving homeomorphisms of $S$ fixing $\Sigma$, up to an
isotopy fixing $\Sigma$.

There is an $\R_{>0}$-action that dilates the atlas of a translation
surface by $ c\in \R_{>0}$. For a stratum $\cH$ and marked stratum
$\cH_{\mathrm{m}}$, we denote the collection of surfaces of arbitrary
area, obtained by
applying such dilations, by $\bar \cH, \, \bar \cH_{\mathrm{m}}$.
The marked stratum $\bar \HH_{\mathrm{m}}$ is a linear manifold modeled on
the vector space $H^1(S, \Sigma; \R^2)$. It has a developing map
$\dev: \bar \HH_{\mathrm{m}} \to H^1(S, \Sigma; \R^2)$, sending an element
of $\bar \HH_{\mathrm{m}}$ represented by $f: S \to M$, where $M$ is a
translation surface, to $f^*(\hol(M, \cdot))$, where for an oriented
path $\gamma$ in $M$ which is either closed or has endpoints at
singularities, $\hol(M, \gamma) = \left(\begin{matrix} \int_\gamma
    dx\\ \int_\gamma dy\end{matrix} \right)$, and $dx, dy$ are the
1-forms on $M$ inherited from the plane. Furthermore, there is an open
cover $\{\mathcal{U}_\tau\}$ of $\HH_{\mathrm{m}}$, indexed by
triangulations $\tau$ of $S$ with triangles whose vertices are in
$\Sigma$, and  maps $\dev|_{\mathcal{U}_\tau}: \mathcal{U}_{\tau} \to
H^1(S, \Sigma; \R^2)$, which are homeomorphisms onto their image, and
such that the transition  maps on overlaps for this atlas are
restrictions of linear automorphisms of $H^1(S, \Sigma; \R^2)$. 

This atlas of charts $\{(\mathcal{U}_\tau,
\dev|_{\mathcal{U}_\tau})\}$ is known as {\em period
  coordinates}. Since each $\mathcal{U}_\tau$ is identified via  period
coordinates with an open subset of the vector space $H^1(S, \Sigma;
\R^2)$, the tangent space at each $\mathcal{U}_\tau$ is identified
canonically with $H^1(S, \Sigma; \R^2)$, and thus the tangent bundle
of $\HH_{\mathrm{m}}$ is locally constant. A sub-bundle of the tangent
bundle is called {\em locally constant} or {\em flat} if it is
constant in the charts afforded by period coordinates. The $\Mod(S,
\Sigma)$-action on $\HH_{\mathrm{m}}$ is properly discontinuous, and
hence $\HH$ is an orbifold, and the map $\pi: \HH_{\mathrm{m}} \to
\HH$ is an orbifold covering map.

The group $G$ acts on translation surfaces in $\HH$ by modifying
planar charts, and acts on $H^1(S, \Sigma; \R^2)$ via its action on
$\R^2$, thus inducing a $G$-action on $\HH_{\mathrm{m}}$. The
$G$-action commutes with the $\Mod(S, \Sigma)$-action, and thus the map
$\pi$ is $G$-equivariant for these actions. The $G$-action on
$\HH_{\mathrm{m}}$ is free, since $\dev(g\q) \neq \dev(\q)$ for any
nontrivial $g \in G$.  We will use the following subgroups of $G$:
\[
  \begin{split}
g_t = \left( \begin{matrix} e^t & 0 \\ 0 & e^{-t}\end{matrix} \right),
\ \ \ 
u_s
= & \left( \begin{matrix}  1 & s \\ 0 & 1\end{matrix} \right) \\
U = \{u_s : s \in \R\}, \ \ \
P = & \left\{ \left( \begin{matrix} a & b
     \\ 0 & a^{-1} \end{matrix} \right): a > 0, \ b
 \in \R\right\}.
\end{split}
\]

\subsection{Rel foliation and real Rel foliation}\label{subsec: rel
  foliation} We define and list some important properties of the Rel
foliation, the real Rel foliation, and the corresponding action on the
space of surfaces without horizontal saddle connections. See
\cite{mahler, eigenform} for more details. See also \cite{zorich
  survey, McM}, and references therein.

We have a canonical splitting $\R^2 = \R \oplus \R$ and we write $\R^2
= \R_{\mathrm{x}} \oplus \R_{\mathrm{y}}$ to distinguish the two
summands in this splitting. There is a corresponding splitting 
\begin{equation}\label{eq: splitting horizontal vertical}
  H^1(S, \Sigma; \R^2) = H^1(S, \Sigma; \R_{\mathrm{x}})  \oplus
  H^1(S, \Sigma; \R_{\mathrm{y}}). 
\end{equation}

We also have a canonical restriction map $\Res: H^1(S, \Sigma; \R^2) \to
H^1(S; \R^2)$ (given by restricting a cochain to absolute
periods). Since $\Res$ is topologically defined, its kernel $\ker (
\Res)$ is $\Mod(S, \Sigma)$-invariant. Moreover, from our convention
that singular points are marked, the 
$\Mod(S, \Sigma)$-action on $\ker ( \Res)$ is
trivial.

Let
\begin{equation}\label{eq: def Z}
  \mathfrak{R}  \df \ker ( \Res ) \ \ \text{ and }  \ \ Z \df \mathfrak{R}
  \cap H^1(S, \Sigma; \R_{\mathrm{x}}). 
\end{equation}

Since $H^1(S, \Sigma; \R_{\mathrm{x}})$ and $H^1(S, \Sigma;
\R_{\mathrm{y}})$ are naturally identified with each other via their
identification with $H^1(S, \Sigma; \R)$, for each $Z_1 \subset Z$ we
can define the space $\mathfrak{R}_1$ spanned by the two copies of $Z_1$
in $H^1(S, \Sigma; \R_{\mathrm{x}})$ and $H^1(S, \Sigma;
\R_{\mathrm{y}})$ respectively. The space $\mathfrak{R}_1$ is the {\em
  complexification} of $Z_1$. This terminology arises from viewing
$H^1(S, \Sigma; \R^2)$ as $H^1(S, \Sigma;
\C)$, a vector space over $\C$, viewing $H^1(S, \Sigma;
\R_{\mathrm{x}})$ and $H^1(S, \Sigma; 
\R_{\mathrm{y}})$ as the real and imaginary subspace of this complex
vector space. With this viewpoint, $\mathfrak{R}_1$ is the $\C$-span
of $Z_1$.

For any subspace $Z_1 \subset \mathfrak{R}$, we can foliate the vector
space $H^1(S, \Sigma; \R^2)$ by affine subspaces parallel to
$Z_1$. Pulling back this foliation using the period coordinate charts
gives rise to a foliation of $\bar {\cH}_{\mathrm{m}}$. Since
monodromy acts trivially on $\mathfrak{R}$, this foliation descends to
a well-defined foliation on $\bar \cH$. It is known (see
e.g.\;\cite[Prop. 4.1]{eigenform}) that the area of a surface is
constant 
on leaves of the Rel foliation, and thus the Rel foliation and any of
its subfoliations descend to a foliation of $\cH$. The foliation
corresponding to $\mathfrak{R}$ (respectively, to $Z$) is known as the
{\em Rel foliation} (respectively, the {\em real Rel foliation}).

Because the $\Mod(S, \Sigma)$-monodromy action fixes all points of
$\mathfrak{R}$, the leaves of the Rel foliation, and any of its
sub-foliations, acquire a translation structure. In particular, they
are equipped with a natural measure. 

For any $v \in Z$ we have a constant vector
field, well-defined on $\HH_{\mathrm{m}}$ and on $\HH$, everywhere
equal to $v$. Integrating this vector field we get a partially defined
{\em real REL flow (corresponding to $v$)} $(t, q) \mapsto
\Rel_{tv}(q)$; the flow may not be defined for all time due to
possible `collide of zeroes'. For every $q \in 
\HH$ it is defined for $t \in I_q$, where the {\em domain of
  definition} $I_q  = I_q(v) $ is an open subset of $\R$ which
contains $0$. The sets $I_q(v)$, are explicitly described in
\cite[Thm. 6.1]{eigenform}. Let $\hat \cH$ denote the set of surfaces
in $\cH$ with no horizontal saddle connections. Then $I_q = \R$ for
all $q \in \hat \cH$.

If $q \in \HH, \, s \in \R$ and $\tau \in I_q$
then
\begin{equation*}\tau \in  I_{u_sq} \ \ \text{ and
  } \ \ \Rel_{\tau v}(u_sq) = u_s \Rel_{ \tau v}(q).
\end{equation*}
Similarly, if $q \in
\HH, \, t \in \R$ and $\tau \in I_q$ then
\begin{equation}\label{eq: commutation rel geodesic}
  \tau' \df e^t \tau \in I_{g_tq} \ \ \text{ and } \ \ 
  \Rel_{\tau' v}(g_t q) = g_t \Rel_{\tau v}(q).
\end{equation}
In particular, since
$P$ preserves $\hat \cH$ and $P = \{g_t u_s : t, s \in \R\}$, there
is an action of $P \ltimes Z$ on $\hat \cH$, given by $(p, z).q = p
\Rel_z(q)$.

\section{Preliminaries from ergodic theory}\label{sec: ergodic theory}
\subsection{Ergodic decomposition}
We will use the notation $\mathbb{G} \circlearrowright (X, \mu)$ to indicate
that $\mathbb{G}$ is a locally compact second countable group, $(X,
\mathcal{B})$ is a standard Borel space, and $\mu$ is a probability
measure on $\mathcal{B}$ preserved by the $\mathbb{G}$-action.
We say that $\mathbb{G}
\circlearrowright (Y, \nu)$ is a {\em factor} of $(X, \mu)$ if there
is a measurable $\mathbb{G}$-invariant conull subset $X_0 \subset X$, and a
measurable map $T: X_0 
\to Y$ such that $T \circ g = g \circ T$ for all $g \in \mathbb{G}$, and $\nu =
T_* \mu$. In this situation we refer to $T$ as the factor map. Given a
factor map, there is a (unique up to nullsets) measure disintegration
$\mu = \int \mu_y \, 
d\nu(y)$, for a Borel mapping $y \mapsto \mu_y$ from $Y$ to the space
of Borel probability measures on $X$, such that $\mu_y(T^{-1}(y))=1$
for $\nu$-a.e. $y$. Equivalently we can write $\mu = \int_x \mu'_x \,
d\mu(x)$, where $\mu'_x \df \mu_{T(x)}$.  For a closed
subgroup $H \subset \mathbb{G}$, we say that $\mu$ is $H$-ergodic if any invariant
set is null or conull. We have the following well-known ergodic
decomposition theorem:
\begin{prop}\name{prop: ergodic components} Suppose $\mathbb{G}
  \circlearrowright (X, \mu)$, and $H$ is a 
  closed subgroup of $\mathbb{G}$. Then there is a factor of $H
  \circlearrowright (X, \mu)$, called the {\em space of ergodic
    components} and denoted  by $X 
\ec H$, with the following properties:
\begin{itemize}
\item[(i)]
  For $\nu$-a.e. $y \in X \ec H$, $\mu_y$ is $H$-invariant and $H$-ergodic.
\item[(ii)]
  $H$ acts trivially on $X \ec H$. 
\item[(iii)]   $H \circlearrowright (X, \mu)$ is ergodic if and only if $X \ec H =
  \{pt.\}.$ 
\item[(iv)]
The properties (i)--(iii) uniquely determine the factor $X \ec H$ up
to measurable isomorphism. 
  \item[(v)]
  If $H \lhd \mathbb{G}$ then $\mathbb{G} \circlearrowright (X \ec H, \nu)$. 

\end{itemize}

\end{prop}
\begin{proof}
For (i) and (ii) see \cite[Thm. 4.4]{Varadarajan} (in the notation of
\cite{Varadarajan},  these assertions follow from the fact that
$\beta$ is a map on points and is $H$-invariant).  Assertion (iii) is
immediate from definitions and (iv) follows from \cite[Lemma
4.4]{Varadarajan}. For (v), one can 
argue using the uniqueness property (iv), and the fact that  the image
of an 
$H$-invariant ergodic measures under any element $g \in \mathbb{G}$ is also
$H$-invariant and ergodic.
  \end{proof}
 \begin{remark}\label{rem:prime}
An action is called {\em prime} if it has no 
factors (besides the action itself, and the trivial action on a point). 
The construction above shows that if $H \lhd \mathbb{G}$, $\mathbb{G}'$ is a
subgroup of $\mathbb{G}$ so that $\mathbb{G}' \circlearrowright (X,\mu)$ is prime and $H
\circlearrowright (X, \mu)$ is not isomorphic to the trivial action,
then $H \circlearrowright(X, \mu)$ is
ergodic. 
This is not the 
approach we will take for proving Theorem \ref{thm: main}. 
\end{remark}
\subsection{Joinings}
We recall some well-known facts about joinings, see \cite{de la rue}
and references therein. 
Let $\mathbb{G} \circlearrowright (X_i, \mu_i)$ for $i=1,2$. A {\em joining} is a
measure $ \theta$ on $X_1 \times X_2$, invariant under the diagonal
action of $\mathbb{G}$, such that $\pi_{i*} \theta =
\mu_i$. A {\em self-joining} is a joining in case $X_1 = X_2$. If
$(X_i, \mu_i) \to (Z, \nu)$ is a joint factor then the {\em relatively
  independent joining over $Z$} is the joining $\int_Z (\mu_1)_z
\times (\mu_2)_z \, d\nu(z)$, where $\mu_i = \int_Z (\mu_i)_z \, d\nu(z)$ is
the disintegration of $\mu_i$. In case $X_1 = X_2 =X$, and $Z = X \ec
H$ is the space of ergodic components of the action of $H$ on $(X,
\mu)$ as in Proposition \ref{prop: ergodic components}, we obtain the {\em
  relatively independent self-joining over $X \ec H$}. This joining satisfies: 

\begin{prop}\name{prop: ergodic theory abstract nonsense}
  The following are equivalent:
  \begin{itemize}
  \item
    $H \circlearrowright (X, \mu)$ is ergodic.
  \item
    The relatively independent self-joining over $X \ec H$ is $\mu
    \times \mu$.
       \end{itemize}

     \end{prop}

     We note two properties of this self-joining. We fix a topology on
     $X$ which generates the $\sigma$-algebra, and denote by $\supp  \, \mu$
the topological support of $\mu$, i.e., the smallest closed set of
full measure.
     \begin{prop}\name{prop: some properties}
Let $\theta$ be the measure on $X \times X$ which is the relatively
independent self-joining over $X \ec H$, for some $H$, and let $T: X
\to X \ec H$ be the factor map. Then the
following hold:
\begin{itemize}
\item
  We have 
  \begin{equation}\label{eq: we have}
\theta = \int_X \mu_{T(x)} \times \mu_{T(x)} \, d\mu(x). 
\end{equation}
\item
  The set $\left \{x \in X : x \notin \supp \, \mu_{T(x)} \right\}$ is of
  $\mu$-measure zero. 
\item
If $X = \supp \, \mu$ then $\supp \, \theta$ contains the diagonal
$\Delta_X  \df \{(x,x) : x 
\in X\}$.
  \end{itemize}
     \end{prop}

     \begin{proof}
   Formula \eqref{eq: we have} is immediate from the definition of the
   relatively independent self-joining over $X \ec H$. Since each
   $\mu'_x = \mu_{T(x)}$ is $H$-invariant and ergodic, and
   $\mu'_x(T^{-1}(T(x)))=1$, the set $\{x \in X: x \notin \supp \,
   \mu'_x\}$ is a nullset. From this, and from \eqref{eq: we have} we
   obtain the last assertion.
              \end{proof}
     
  \subsection{Ergodicity, mixing, and mixing of all
    orders} \label{subsec: mixing of all orders}
  For $\mathbb{G}
  \circlearrowright (X, \mu)$, let 
$L^2_0(\mu)$ denote the Hilbert space of $L^2$-functions on $(X, \mu)$
of integral zero, and let $k \geq 2$. The action   is called
{\em $k$-mixing} if for any $f_1, \ldots, f_k \in L^2_0(\mu)$ and for
any $k-1$ sequences $\left(g^{(i)}_n \right)_{n \in \N} \in
\mathbb{G}, \ i=1, \ldots, k-1$, for which all of the sequences
$$
\left(g^{(i)}_n \right)_{n \in \N}\ \ ( 1 \leq i \leq  k-1) \ \ \ \text{ and } \left(g^{(i)}_n
(g^{(j)}_n)^{-1} \right)_{n \in \N}\ \  (1 \leq i < j \leq k-1)
$$
eventually leave every compact subset of $\mathbb{G}$, we have
$$
\int_X f_1\left(g^{(1)}_nx \right) \cdots f_{k-1}\left(g_n^{(k-1)}x \right) f_k(x) \, d\mu(x)
\stackrel{n \to \infty}{\longrightarrow} \prod_{i=1}^k \int_X f_i \, d\mu.
$$
We say that the action is {\em mixing} if it is 2-mixing, and {\em
  mixing of all orders} if it is mixing of order $k$ for any $k \geq
2$. It is easy to check that mixing implies ergodicity of any
unbounded subgroup of $\mathbb{G}$. We have the 
following:

  \begin{prop}\label{prop: adapting Mozes}
Let $Z_0 \cong \R$ and let $\{g_t\}$ be a one-parameter group acting
on $Z_0$ by dilations, i.e., for all $v \in Z_0$ and $t \in \R$ we
have 
$g_t v = e^{\lambda t}v$ for some $\lambda \neq 0$. Let $F = \{g_t \}
\ltimes Z_0$ and let $F \circlearrowright (X, \mu)$ be a probability
space. The following are equivalent:
\begin{enumerate}[(a)]
\item \label{item: first item} the restricted flow
  $Z_0\circlearrowright (X, \mu)$ is ergodic;  
\item \label{item: second one}  the restricted flow
  $Z_0\circlearrowright (X, \mu)$ is mixing of 
  all orders; 
\item \label{item: 2.5} the restricted flow  $Z_0\circlearrowright (X,
  \mu)$ is mixing; 
\item \label{item: third one} any nontrivial element of $Z_0$ acts ergodically. 
\end{enumerate}
 
\end{prop}

\begin{remark}
The group $F$ appearing in Proposition \ref{prop: adapting Mozes} is
isomorphic as a Lie group to  the subgroup $P$ of upper triangular
matrices in $G$, but in our application we will use it for the group
generated by a one-parameter real Rel flow $Z_0$ and the diagonal flow $\{g_t\}$.
\end{remark}

\begin{proof}
Clearly \ref{item: second one} $\implies$ \ref{item: 2.5}
$\implies$ \ref{item: third one} $\implies$ \ref{item: first item}. We
assume that the $Z_0$-flow is ergodic. 
  To see that it is mixing, it is enough by
  \cite[Chap. 2, Prop. 5.9]{Petersen}
  to prove that it has countable Lebesgue spectrum, and for this, use
  \cite[Prop. 1.23 \& Prop. 2.2]{Katok_Thouvenot}.  The proof of
  mixing of all orders follows verbatim from an argument of
  Mozes \cite{Mozes}, for 
  mixing actions of Lie groups which are `Ad-proper'. Since our group
  $ F$ is not Ad-proper, we cannot cite \cite{Mozes}
  directly, so we sketch the proof. For notational convenience we
  deduce 3-fold mixing from mixing (the proof that `$k$-fold mixing
  $\implies  k+1$-fold mixing', for $k \geq 3$, is identical but
  requires more cumbersome notation). 

  We use additive notation in the
      group $Z_0$, and denote the action of $Z_0$ on $X$ by $(z,x)
      \mapsto z.x$. 
Let 
  $\left(b_n\right)_{n \in \N}$ and $\left(c_n \right)_{n \in \N}$ be sequences in
      $Z_0$ such that each of the sequences $\left(b_n\right)_{n \in
        \N}, \, \left(c_n\right)_{n \in \N}, \, \left(b_n+
        c_n\right)_{n \in \N} $ eventually leaves 
      every compact set,  and let $f_1,
      f_2, f_3$ be in $L^2_0(\mu)$. We need to prove that
      $$
\int_X f_1(x) f_2(b_n.x) f_3((b_n+c_n).x) \, d\mu(x) \stackrel{n \to
  \infty}{\longrightarrow}  \int_Xf_1 \, d\mu \,  \int_X f_2 \, d\mu \, \int_X
f_3  \, d\mu. 
      $$
      For each $n$, define a measure $\mu_n$ on $X^3 \df X \times X \times X$ by
      $$
      \int_{X^3}  f \, d\mu_n \df \int_X f(x, b_n.x, (b_n+c_n).x) \, d\mu(x),
      \ \ 
\forall f \in C_c(X^3).
      $$
That is, $\mu_n$ is the pushforward of the diagonal measure on
$X^3$ by the triple $(0, b_n, b_n+c_n)$. It is easy to see that
3-mixing is equivalent to the fact that the weak-* limit of $\mu_n$ is
the measure 
$\mu^3 \df \mu \times \mu \times \mu$. The group $F^3 \df F\times F
\times F$ acts on $X^3$ by acting separately on each component, and as
in \cite{Mozes}, since $Z_0$ is mixing,     
it suffices to show that any measure $\nu$ on $X^3$ which is a weak-*
limit of a subsequence of $\left(\mu_n\right)_{n \in \N}$, is
invariant under $(0, u, v) \in \R^{3} \subset F^3,$ for some $(u,v)
\in \R^{2} \sm  (0,0)$. We claim that 
for any $s \in \R$ the measure 
$\mu_n$ is invariant under 
$$h_n(s)\df \left(g_s , b_n \cdot g_s \cdot (-b_n),
(b_n+c_n) \cdot g_s \cdot (-b_n-c_n) \right),$$
where the multiplication is in the group $F^3$. 
 Indeed,  since $\mu$ is $\{g_s\}$-invariant,
$$\int_{X^3}f\, d\mu_n=\int_{X}f \left(g_sx,b_n. (g_sx), (b_n+c_n).(g_sx)\right) \,
d
\mu (x), $$
and
$$
h_n(s) 
\cdot (\mathrm{id}_F, b_n ,b_n+c_n)= (g_s , b_n \cdot
g_s, (b_n+c_n) \cdot g_s).$$  That is, applying $h_n (s) $ 
changes one description of $\mu_n$ to another.

We embed $F$ as a multiplicative group of matrices in $\GL_{2}(\R)$
and let $d_F$ be the metric  on $F$ induced by some norm on $\GL_{2}(\R) $. 
By a straightforward computation we have
$$
h_n(s) = \left(g_s,\ (1-e^{\lambda s})b_n \cdot g_s, \ (1-e^{\lambda
    s})(b_n+c_n)\cdot g_s \right),
  $$
  and $d_F(\mathrm{id}_F,
h_n(s_n))$ is a continuous function of $s$ which goes to $0$ as $s \to
0$ and for any fixed $s>0$, increases to infinity as $n \to \infty$. 
Therefore  we can choose $s_n \to 0$ so that
$d_F(\mathrm{id}_F, 
h_n(s_n))=1$ for all large enough $n$. As in \cite{Mozes}, $\nu$ is
invariant under some subsequential limit of $h_n(s_n)$ which is of the
form $(0, u, v)$ for some $(u, v) \in \R^{2} \sm (0,0)$. This establishes our
sufficient condition.  
\end{proof}


\section{The relatively independent self-joining for a Rel
  flow}   \label{sec: Mautner}
Recall that $\hat\cL \subset \cL$ is the set of surfaces without
horizontal saddle connections, and this is a $P$-invariant set of full
measure with respect to $m_{\cL}$. We can
combine the product action of $Z_\cL \times Z_\cL$ on $\hat\cL \times
\hat\cL$
with the diagonal action of $P$ to obtain an action of the semi-direct
product $P \ltimes (Z_\cL
\times Z_\cL)$ on $\hat\cL \times \hat\cL$. Since $\hat\cL \subset \cL$ is of
full measure, and the arguments of this section involve passing to
sets of full measure, in the remainder of this section we will ignore
the distinction between $\cL$ and $\hat\cL$.

\begin{prop}\name{prop: Mautner}
  Let $Z \subset Z_\cL$ be a closed connected subgroup. 
If $\theta$ is an invariant probability measure for an action of the semidirect
product $P \ltimes (Z\times Z)$ on $\cL \times \cL$, then any $f
\in L^2(\theta)$ which is 
$\{g_t\}$-invariant is also $Z \times Z$-invariant. 
\end{prop}

\begin{proof}
 For 
any $z \in Z \times Z$, $g_t z g_{-t} \to_{t \to -\infty} 0$. So the claim
follows from the Mautner phenomenon, see e.g. 
\cite[Prop 11.18]{EW}.
\end{proof}

\begin{prop}\name{prop: ergodicity}
Let $(\cL, m_\cL)$ be a $G$-orbit-closure with a fully supported
$
P$-invariant ergodic measure, let
$Z \subset Z_\cL$ be a connected closed subgroup, and
let $\theta$ on
$\cL\times \cL$
be the relatively independent joining
over $\cL \ec Z$. Then $\theta$ is $P$-invariant and $\{g_t\}$-ergodic
(and hence $P$-ergodic).  Also $\Delta_{\cL} \subset \supp \,
\theta$. 
\end{prop}

As we will see in \S \ref{sec: magic wand upgraded}, under the
conditions of the Proposition, $m_\cL$ is the so-called `flat
  measure' on $\cL$. 
\begin{proof}
  Let $\pi:
  \LL \times \LL \to \cL$ be the projection onto
  the first factor, and let $\nu
  = \pi_* \theta.$ For each $x 
  \in \cL$, let $\Omega_x \df  \pi^{-1}(x) = \{x\} \times \cL$ be the fiber, and
let $\theta_x$ be the fiber measure appearing in the disintegration
$\theta = \int_{\cL } \theta_x \, d\nu(x)$. Then $Z$ acts on
$\Omega_x$ via the second factor in $Z \times Z$, and $\theta_x$ is 
$Z$-invariant and ergodic by the definition of the ergodic
decomposition.

It follows from Proposition \ref{prop: ergodic components}(v) that
$\theta$ is $P$-invariant. To prove ergodicity,  
let $f \in L^2(\cL \times \cL, \theta)$ be a $P$-invariant function. By
Proposition \ref{prop: Mautner}, $f$ is $Z \times Z$-invariant. For each
$x \in \cL$, let $f_x \df f|_{\Omega_x}$. There is $\cL_0 \subset \cL$
such that $m_{\cL}(\cL_0)=1$ and for every $x \in \cL_0$, $f_x $
belongs to $
L^2(\Omega_x, \theta_x)$ and is $Z$-invariant. Hence, by ergodicity,
there is $\bar{f}: \cL_0 \to \R$ such that for every $ x\in \cL_0$,
$\bar{f}(x)$ is the $\theta_x$-almost-sure value of $f_x$. Since $f$
is $P$-invariant for the diagonal action of $P$, $\bar{f}$ is
$P$-invariant for the action of $P$ on $\cL$. By ergodicity of $P \circlearrowright 
(\cL, m_\cL)$, $\bar f$ is $\nu$-a.e. constant, and thus $f$ is
$\theta$-a.e. constant.

The last assertion follows from Proposition \ref{prop: some properties}. 
\end{proof}
\section{An upgraded magic wand theorem}\label{sec: magic wand upgraded}
The
celebrated `magic wand' Theorem of 
Eskin and Mirzakhani \cite{EM}, and 
ensuing work of Eskin, Mirzakhani and Mohammadi \cite{EMM}, classified
$P$- and $G$-invariant probability measures and orbit-closures on
strata of translation surfaces. These results can be summarized as
follows (see \cite[Defs. 1.1 \& 1.2, Thms. 1.4 \& 1.5]{EM}):

\begin{thm}\name{thm: magic wand} Let $\cH, \, \cH_{\mathrm{m}}, \,
  \bar \HH, \, \bar \HH_{\mathrm{m}}$ be as in \S \ref{subsec: strata}. Any
  $P$-invariant ergodic probability measure 
  $m$ has the following
  properties:
  \begin{itemize}
  \item[(i)]
    It is $G$-invariant.
  \item[(ii)]
    There is a complex-affine manifold $\mathcal{N}$ and a proper immersion
    $\varphi: \mathcal{N} \to \bar \HH$ such that
    $$\cL \df \supp \, m = \HH
    \cap \varphi(\mathcal{N}).$$
  \item[(iii)]
    There is an open $G$-invariant subset $U \subset \bar \HH$
    satisfying $m(U )=1$, and for any $x \in U \cap \cL$ there
    is an open set $V$ containing $x$ such that $V$ is evenly covered
    by $\mathcal{V} \subset \HH_{\mathrm{m}}$ under the map $\pi:
    \bar{\HH}_{\mathrm{m}} \to \bar{\HH}$, 
    and $\psi \df \dev \circ 
    (\pi|_{\mathcal{V}})^{-1} \circ \varphi$ coincides on its domain
    with a $\C$-linear map, with real coefficients.
  \item[(iv)]
    The subspace $W\df \mathrm{Im}
  (\psi)$ is symplectic, and the measure $m$ is
  obtained via the cone construction from the Lebesgue measure on
  $W$.
\item[(v)]
 The complement $\cL \sm U$ is a finite union of supports of measures
 satisfying properties (i)--(iv), for which the manifolds
 $\mathcal{N}'$ appearing in (ii) satisfy $\dim \mathcal{N}' < \dim \mathcal{N}.$
    \end{itemize}

  Any orbit-closure for the $P$-action is a set $\cL$ as above. 
  \end{thm}

  We will refer to $\cL$ as an {\em orbit-closure} and to $m = m_\cL$
  as a {\em flat measure} on $\cL$. Orbit-closures are referred
  to as {\em affine invariant manifolds} and also as {\em invariant
    subvarieties}. The use of an evenly covered neighborhood in item
  (iii) is a standard approach for defining period coordinates  (see e.g. \cite{MS}). We
  refer to \cite{Wright survey}  
  for a survey containing more information on orbit-closures.

  In a forthcoming work of Brown, Eskin, Filip and Rodriguez-Hertz,
  the same conclusion is obtained for the
diagonal actions  of 
$P$ and $G$ on a product of strata $\HH \times \HH'$. Namely,  the
following is shown:

\begin{thm}\name{thm: super duper wand}
Let $\HH, \HH'$ be strata of translation surfaces, and let $P$ and
$G$ act on $\HH \times \HH'$ via their diagonal embeddings in $G
\times G$. Then all of the conclusions of Theorem \ref{thm: magic
  wand} hold for this 
action (with $\bar \cH\times \bar \cH'$ replacing $\bar \cH$). 
  \end{thm}

  \section{Proof of main result}\label{sec: putting together}
Using Theorem \ref{thm: super duper wand} and further work of Wright
\cite{Wright field of definition}, 
we can prove our main result.

\begin{proof}[Proof of Theorem \ref{thm: general}]
  Let $Z_0 = \spa_{\R}(z_0)$ be a one-dimensional connected real Rel subgroup.
  Assume that \eqref{item: one} fails, so that the action of $Z_0$ on 
$\left(\cL, m_{\mathcal{L}} \right)$ is not mixing of all
orders.  Then, by Proposition \ref{prop: adapting Mozes} it is not
ergodic. Let $\theta$ be the relatively independent self-joining over
$\cL \ec Z_0$.  Applying Propositions \ref{prop: ergodic theory
  abstract nonsense} and \ref{prop: some properties} we have that $\theta \neq 
m_{\cL} \times m_{\cL}$ and $\Delta_{\cL} \subset \supp \,
\theta$. Applying Proposition \ref{prop:
  ergodicity} and Theorem 
\ref{thm: super duper wand}, we have that there is a $G$-invariant
open subset $U$ of
full $\theta$-measure such that 
$U \cap \supp \, 
\theta$ is the isomorphic image of an affine complex-linear manifold whose
dimension is strictly smaller than $2 \dim \bar \cH$, and $\theta$ is
obtained from Lebesgue measure on this complex-linear manifold by the
cone construction.

We claim that the set
$$U_1 \df \{q \in \cH: (q,q) \in U\}
$$
is of full measure for $(\pi_{1})_* \theta$, where 
$\pi_1 : \cL \times \cL
\to \cL$ is the projection onto the first factor. Indeed, the measure
$\theta$ is invariant under $Z_0 \times \{ \mathrm{Id} \}$, and hence
so is its support. Since $Z_0$ acts by homeomorphisms where defined,
and using property
(v) in Theorems \ref{thm: magic wand} and \ref{thm: super duper
  wand}, we have that the set $U$ is also $Z_{0} \times
\{\mathrm{Id}\}$-invariant. Thus for any $Z_{0}$-ergodic measure,
it is either null or conull. Thus if $q \notin U_1$ and $q$ is
generic for the measure $\mu_{T(q)}$ appearing in \eqref{eq: we have},
then  $\mu_{T(q),q}$ assigns measure zero to 
$U$, where $\mu_{T(q),q}$ is the measure on $\supp\, \theta$ defined
by $\mu_{T(q),q}(A)=\mu_{T(q)}(\{q':(q',q) \in A\})$. If this were to
happen for a positive measure of 
$q$ it would follow from \eqref{eq: we have} and the fact that
$\mu_{T(q)}\times\mu_{T(q)}=\int \mu_{T(q'),q'}d\mu_{T(q)}$ that $U$
does not have 
full measure for $\theta$. 

For $q \in U_1$, let
$N_q$ denote the connected  component of $
U \cap \pi_1^{-1}(q) \cap \supp \, \theta
$ containing $(q,q)$.
Since the fibers
$\pi_1^{-1}(q)$ are also affine submanifolds of $\cL \times \cL$, we
have that the $N_q$ are affine submanifolds contained in $\pi_1^{-1}(q)
\cong \cL$, so we can identify them with invariant submanifolds in
$\cL$ (which we continue to denote by $N_q)$. With this notation we
have $q \in N_q$. 

The mapping $q \mapsto T(N_q)$ is locally constant; that is,
letting $V \subset \bar \cH$ and $\mathcal{V} \subset \bar
\cH_{\mathrm{m}}$ be open sets such that $\pi|_{\mathcal{V}}:
\mathcal{V} \to V$ is a homeomorphism and $q \in V$, the map $q
\mapsto \dev \circ \pi|_{\mathcal{V}}^{-1} (q) $ sends a neighborhood
of $q$ in $N_q$ to an affine subspace $W_q$ of $H^1(S, \Sigma ;
\R^2)$, and the corresponding linear spaces $W_q - W_q$ are the same for
all $q \in V$. Since $m_{\cL} \times m_{\cL}$ is the unique $P$-invariant ergodic
measure on $\cL \times \cL$ of full support, we have $\dim N_q < \dim
\cL$ for every $q \in U_1$. 

Let
$\bar N_q$ denote the set of surfaces (not necessarily of area one)
which are obtained by rescaling surfaces in $N_q$, and let
$$
\mathfrak{N}_q \df T_q(\bar N_q )
$$
(the tangent space to $\bar N_q$ at $q$, thought of as a subset of the
tangent space $T_q(\bar \cL)$). The assignment
$q\mapsto \mathfrak{N}_q $ 
defines a proper flat sub-bundle of the tangent bundle $T(\bar \cL)$. 
Flat sub-bundles of $T(\bar \cL)$ were classified in \cite{Wright field of
  definition}. According to
\cite[Thm. 5.1]{Wright field of definition}, 
$\mathfrak{N}_q \subset \mathfrak{R}_{\cL}$ for each $q$, and
$\mathfrak{N}_q$ is a complex linear subspace which is locally
constant. Since $\mathfrak{R}_{\cL}$ is acted on trivially by monodromy,
we in fact have that $\mathfrak{N}_q$ is independent of $q$, and we
denote it by $\mathfrak{R}$. The leaves $\mathfrak{R}(q)$ are
contained in $\bar N_q$ for each $q$, and of the same dimension. That
is, $\mathfrak{R}(q)$ is the connected component of $\bar{N}_q$
containing $q$. Since Rel deformations do not affect the
area of the surface, we see that $\bar{N}_q = N_q$. In particular 
$\mathfrak{R}(q)$ is closed for each $q$. 

By Proposition \ref{prop: some properties}, for a.e. $q$, $N_q$ is the support
of the ergodic component $(m_{\cL})_q$, and in particular
$$
(m_{\cL})_q(N_q)<\infty, \ \  \text{ for a.e. } q.
$$
Since $Z_0$ acts
ergodically with respect to  $(m_{\cL})_q$, we have that almost surely
$N_q =\mathfrak{R}(q)$. Since the measure $(m_{\cL})_q$ is 
affine in charts, it is a scalar multiple of the
translation-invariant measure on 
$\mathfrak{R}(q)$, and thus the volume $V_q$ of  $\mathfrak{R}(q)$ (with
respect to its translation-invariant measure) is almost surely
finite. It is clear that the function $q \mapsto V_q$ is
$U$-invariant, and by ergodicity, it is constant almost surely.
\end{proof}

\begin{remark}
We note that the above argument works under much weaker conclusions
than those given in Theorem \ref{thm: super duper
  wand}. Indeed, in the first step of the argument, Theorem \ref{thm: super duper
  wand} was used simply to extract a $G$-invariant assignment $q
\mapsto N_q$, where $N_q$ 
is a subspace of $T_q(\cL)$, which is proper if $\theta$ is not the
product joining. A fundamental fact about such $G$-invariant assignments is
that they are very restricted -- besides \cite{Wright field of
  definition}, see  \cite{EFW} and \cite{Filip}.  In particular,
\cite{Filip} gives strong restrictions on assignments that are only
assumed to be defined almost everywhere and 
measurable. 
  \end{remark}

  \section{A topological condition for Rel ergodicity}\label{sec:
    topological condition}
Let $Z_0 \subset Z$ be a subspace. 
We say that a translation surface $x$ is \textit{$Z_0$-stably 
  periodic} if it can be presented as a finite union of horizontal cylinders and the
$Z_0$-orbit of $x$ is well defined. Recall that a {\em horizontal
separatrix} is a horizontal leaf whose closure contains at least one
singularity, and it is a {\em horizontal saddle connection} if  its
closure contains two singularities. 
Then the condition of being $Z_ 0$-stably periodic is equivalent to
requiring that all horizontal separatrices starting 
at singular points are on horizontal saddle connections, and $Z_0$ preserves the
holonomy of every 
horizontal saddle connection on $x$.  In case $Z =Z_0$ is the full
real Rel group, we say that $x$ is {\em fully stably 
  periodic.} This is equivalent to saying that all horizontal separatrices starting
at singular points are on saddle connections, and all horizontal
saddle connections start and end at the same singularity. In
particular, for any cylinder $C$ on a fully stably periodic surface,
each boundary component of $C$ is made of saddle connections starting
and ending at the same singular point $\xi$; we say that the boundary
component {\em only sees singularity $\xi$}. For more information on
the real Rel action on surfaces which are horizontally completely
periodic, see \cite[\S 6.1]{HW}. 

\begin{prop} \label{prop: criterion}
Suppose $x$ is a surface which is $Z_0$-stably 
 periodic,
and $v \in Z_0$ moves two singularities $p$ and $q$ with respect to each
other. Suppose that $x$ contains two cylinders $C_1$ and $C_2$ 
that both only see singularity $p$ on one boundary component and only
see singularity $q$ on another boundary component. Finally suppose the
circumferences $c_1, c_2$ of these 
cylinders satisfy $\frac{c_1}{c_2} \notin \Q$. Then Case  \eqref{item: two}
of Theorem \ref{thm: main} does not hold for
$x$.  
  \end{prop}

  \begin{proof}
Since $\frac{c_1}{c_2} \notin \Q$, the trajectory $\{\rel_{tv}(x) : t \in
\R\}$ is not closed, let $\mathcal{L}$ denote its closure. We claim
that the tangent space to $\mathcal{L}$ is not contained in $Z$. Let
$\sigma_1$ denote a saddle connection from $p$ to $q$ in $C_1$ and
let $\sigma_2$ denote a saddle connection from $q$ to $p$ in
$C_2$. Let $\sigma $ be the concatenation. Then $\sigma$ represents an
absolute homology class because 
it goes from $p$ back to $p$, and it is nontrivial because the
vertical component of its holonomy on $x$ is nonzero.
If we consider the restriction of the rel-action to $C_1 \cup C_2$ then it only affects
the twist parameters, which is a 2-dimensional space. This space can
be generated by the horizontal holonomy of $\sigma_1$ and the
horizontal holonomy of $\sigma_2$. Since 
$\frac{c_1}{c_2} \notin \Q$, this 
restricted action does not give a closed orbit. So the tangent space
to $\mathcal{L}$
contains directions, which continuously affect the holonomy of
$\sigma$. Since $\sigma$ is an absolute period, we see that the
tangent space to $\mathcal{L}$ is not contained in $Z$. 
    \end{proof}

    \section{Checking the condition for strata}\label{sec: checking condition}
     Let $\cH = \cH(a_1, \ldots, a_k)$ and for $i, j \in \{1, \ldots,
    k\}$, let $\xi_i, \xi_j$ be the corresponding singular points of a
    surface in $\cH$. Let $z \in \mathfrak{R}$ be a Rel cohomology
    class. We say that {\em $z$ moves $\xi_i, \xi_j$ with respect to
      each other} if for some (equivalently, every) $\alpha \in H_1(S,
    \Sigma)$ represented by a path starting at $\xi_i$ and ending at
    $\xi_j$, we have $z(\alpha) \neq 0$. Below when we discuss a
    stratum $\mathcal{H}(a_1, \ldots, a_k)$ we 
    allow $a_i=0$, that is we allow marked points. We call points with
    cone angle $2\pi$ (that is, with $a=0$) {\em removable
      singularities}, and otherwise we call them {\em
      non-removable}. The following result, which clearly implies
    Theorem \ref{thm: main}, 
    allows strata with removable singularities.
    
    \begin{thm}\label{thm: main 1}
Let $\HH$ be a connected component of a stratum $\HH(a_1, \ldots,
a_k)$. Let $m_\HH$ be
the Masur-Veech measure on $\HH$, let $Z$ be the
corresponding real Rel foliation, and let $Z_0 \subset Z$ be a
one-dimensional connected subgroup of $Z$. Suppose that there are $1 \leq
i < j \leq k$ with corresponding singular
points $\xi_i, \xi_j$, such that $a_i>0, \, a_j>0$ and such that some
element of $Z_0$ moves $\xi_i, \xi_j$ with respect to each other. Then
the $Z_0$-flow on 
$(\HH, m_{\HH})$ is 
mixing of all orders (and in 
particular, ergodic). 
\end{thm}
   
 Clearly, Theorem \ref{thm: main 1} follows from Theorem \ref{thm: general},
 Proposition \ref{prop: criterion}, and the following result.

\begin{prop}\label{prop:verifying crit}
Let $\mathcal{H} \subset \mathcal{H}(a_1,\ldots,a_k)$ be a connected
component of a stratum of translation surfaces with at least two
non-removable singular points. If $p\neq q$ is any pair of
non-removable singularities then there exists $M \in \mathcal{H}$,
which has cylinders $C_1,C_2$ with circumferences $c_1,c_2$ so that  
\begin{enumerate}
\item \label{item: first} $M$ is fully stably periodic.
\item \label{item: second} $\frac{c_1}{c_2} \notin \mathbb{Q}$.
\item \label{item: third}  Both $C_1$ and $C_2$ only see singularity
  $p$ on one boundary 
  component and only see singularity $q$ on the other boundary
  component. 
\end{enumerate}

\end{prop}

  For the proof of Proposition \ref{prop:verifying crit} we will also
  need the following: 
\begin{prop}\label{prop:first step}
Let $\cH = \mathcal{H}(a_1,\ldots,a_k)$ be a stratum of translation surfaces
with at least two 
singular points (that is $k\geq 2$). If $p\neq q$ is any pair of distinct
singularities (possibly removable), then there exists $M \in
\mathcal{H}$, so that $M$ is fully stably 
periodic and there exists a cylinder on $M$ that only sees singularity
$p$ on one boundary component, and only sees singularity $q$ on the
other boundary component. 
\end{prop}

Propositions \ref{prop:verifying crit} and \ref{prop:first step} will
both be proved by induction, after 
some preparations. 

\begin{lem}[The basic surgery -- gluing in a torus] \label{lem: basic surgery}
  Let
  $\cH = \mathcal{H}(b_1, \ldots ,b_{\ell})$ be a 
  stratum of translation surfaces, and let $M \in
  \cH$, with singularities labeled by
  $\xi_1, \ldots, \xi_\ell$, so that the order of $\xi_i$ is
  $b_i$. Suppose $M$ has a horizontal cylinder $C$, with
  circumference $c$, where one boundary component is made of saddle
  connections that begin and end at $\xi_i$, and the other is made of
  saddle connections that begin and end at 
  $\xi_j$, where $b_i \geq 0$ and $b_j \geq 0$ (so that $\xi_i, \xi_j$
  might be removable). Then for all $w>0$ there exists $M' \in
  \mathcal{H}(b_1, \ldots ,b_{i}+1, \ldots,b_{j}+1, \ldots,b_\ell)$,
  with singularities labeled $\xi'_1, \ldots, \xi'_\ell$,  which has two
  horizontal cylinders $C'_1$, $C'_2$, where $C'_1$ has circumference $c+w$ and
  $C'_2$ has circumference $w$. The complements $M \sm C$ and $M' \sm
  (C_1 \cup C_2)$ are isometric, by an isometry mapping $\xi'_i$ to
  $\xi_i$ for all $i$.  The cylinders $C_1$ and $C_2$ only see
  singularity $\xi'_i$ on one boundary component, and $\xi'_j$ on
  another. Moreover, 
  if $M$ is fully stably periodic then so is $M'$. 
\end{lem}

\begin{proof}
  It will be easier to follow the proof while consulting Figures
  \ref{fig: first surgery, before} (before) and \ref{fig: first
    surgery, after} (after).
\begin{figure}
\begin{tikzpicture}
\draw (0,0)--(6,0)--(6,3)--(1,3)--(1.3,5)--(.3,5)--(0,3)--(0,0);
\node at (6,3) [circle, draw, fill=white, outer sep=0pt,  inner
sep=1.5pt]{};
\node at (6,0) [circle, draw, fill=black, outer sep=0pt,  inner
sep=1.5pt]{};
\node at (0,3) [circle, draw, fill=white, outer sep=0pt,  inner
sep=1.5pt]{};
\node at (0,0) [circle, draw, fill=black, outer sep=0pt,  inner
sep=1.5pt]{};
sep=1.5pt]{};
\node at (0,1.5) {$\triangle$};
\node at (6,1.5) {$\triangle$};
\draw[<->,dashed] (.05,1)--(5.95,1);
\node at (3,1.4) {$c$};
\end{tikzpicture}
\caption{The surface $M$ has a cylinder of circumference $c$, and its
  boundary components see only the singularities $\xi_i$ and $\xi_j$
  (denoted by $\circ$ and $\bullet$). The edges not labeled by
  $\triangle$ are connected to $M \sm C$. }\label{fig: first surgery, before}
\end{figure}
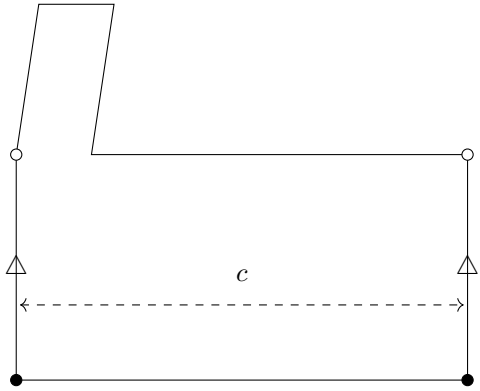
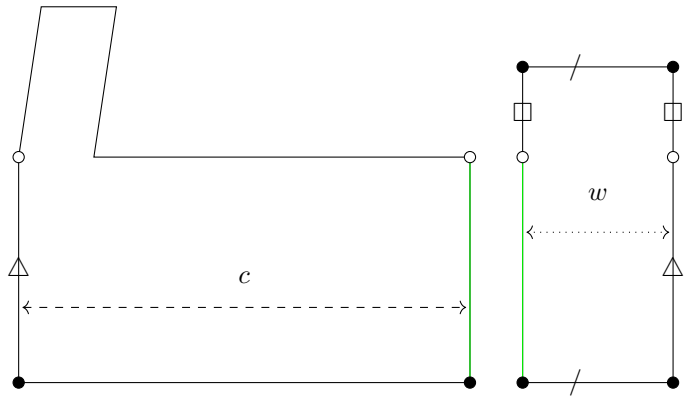
\begin{figure}
\begin{tikzpicture}
\draw (0,0)--(6,0)--(6,3)--(1,3)--(1.3,5)--(.3,5)--(0,3)--(0,0);
\draw (6.7,0)--(6.7,4.2)--(8.7,4.2)--(8.7,0)--(6.7,0);
\draw[green] (6,0)--(6,3);
\draw[green] (6.7,0)--(6.7,3);
\node at (6,3) [circle, draw, fill=white, outer sep=0pt,  inner
sep=1.5pt]{};
\node at (8.7,0) [circle, draw, fill=black, outer sep=0pt,  inner
sep=1.5pt]{};
\node at (8.7,3) [circle, draw, fill=white, outer sep=0pt,  inner
sep=1.5pt]{};
\node at (6,0) [circle, draw, fill=black, outer sep=0pt,  inner
sep=1.5pt]{};
\node at (8.7,4.2) [circle, draw, fill=black, outer sep=0pt,  inner
sep=1.5pt]{};
\node at (6.7,4.2) [circle, draw, fill=black, outer sep=0pt,  inner
sep=1.5pt]{};
\node at (6.7,3) [circle, draw, fill=white, outer sep=0pt,  inner
sep=1.5pt]{};
\node at (6.7,0) [circle, draw, fill=black, outer sep=0pt,  inner
sep=1.5pt]{};
\node at (0,3) [circle, draw, fill=white, outer sep=0pt,  inner
sep=1.5pt]{};
\node at (0,0) [circle, draw, fill=black, outer sep=0pt,  inner
sep=1.5pt]{};
\draw[<->,dashed] (.05,1)--(5.95,1);
\node at (3,1.4) {$c$};
\draw[<->,dotted] (6.75,2)--(8.65,2);
\node at (7.7,2.5) {$w$};
\node at (7.4,0) {/};
\node at (7.4,4.2) {/};
\node at (0,1.5) {$\triangle$};
\node at (8.7,1.5) {$\triangle$};
\node at (8.7,3.6){$\square$};
\node at (6.7,3.6){$\square$};
\end{tikzpicture}
\caption{To obtain $M'$ from $M$, glue in a torus (rectangle on the
  right). This transforms $C$ into a 
  cylinder $C'_1$ of circumference $c+w$, and adds a horizontal
  cylinder $C'_2$ of
  circumference $w$. Edges not labeled by 
  $\triangle$, $\square$, / or  the color green are attached to $M'
  \sm (C'_1 \cup C'_2)$. \label{fig: first surgery, after}}
\end{figure}
Given a polygonal presentation for $M$, we give a polygonal 
presentation for $M'$. Let $M$ be a polygon representation for $M$ in 
which the cylinder $C$ is represented by a parallelogram $P$ (in 
Figure \ref{fig: first surgery, before}, the large rectangle in the 
center of the presentation), with two 
horizontal sides of length $c$, non-horizontal sides identified to 
each other, and the singular points $\xi_i$, $\xi_j$ on adjacent 
corners of $P$. Thus the non-horizontal sides of $P$ represent a saddle 
connection $\sigma$ on $M$ connecting $\xi_i$ to $\xi_j$. We consider 
the two non-horizontal sides of $P$ as distinct and label them by 
$\sigma_1, \sigma_2$. Let $P'$ be a 
parallelogram with sides parallel to those of $P$, where the 
horizontal sides have length $w$ and the nonhorizontal sides are 
longer than the ones on $P$ (in Figure \ref{fig: first surgery,
  after}, $P'$ is to the right of  
$P$).

Label the two horizontal sides of $P'$ by
$h'_1$ and $h'_2$, and identify them by a translation. Partition the non-horizontal
sides of $P'$ into two segments. The segments $\sigma'_1, \sigma'_2$
are parallel to each other and have the same length as $\sigma_1,
\sigma_2$, and start at a corner of $P$. The segments $\gamma'_1,
\gamma'_2$ comprise the remainder of the non-horizontal sides of $P'$
(and in particular, have the same length). Identify $\gamma'_1$ to
$\gamma'_2$ by a translation, and identify $\sigma'_1, \sigma'_2$ to
$\sigma_1, \sigma_2$ by a translation so that each
$\sigma'_i$ is attached to the $\sigma_j$ with the opposite
orientation. Let $M'$ be the translation surface corresponding to
this presentation. It is clear that $M'$ has the required properties.

\end{proof}

\begin{proof}[Proof of Proposition \ref{prop:first step}]
  The proof is by induction on $\sum a_i$. 
  
\textbf{Base of induction:} 
  The base case is the stratum $\HH(a_1,
0^{s})$, that is, one singular point (removable or non-removable) of
order $a_1$, and some number $s \geq 1$ of removable singular
points. In this case we take 
a surface in $\HH(a_1)$ which is made of one horizontal cylinder. We
label the singular point by $\xi_1$ and 
place additional removable singular points $\xi_2, \ldots, \xi_{s+1}$ in
the interior of the cylinder, at different heights (so that the
resulting surface has no horizontal saddle connections between
distinct singularities) and so that $\xi_i$ and $\xi_j$ are on
opposite sides of a cylinder.

\medskip

\textbf{Inductive step:} 
Suppose $\HH' = \HH(a_1, \ldots, a_k)$ is our stratum, where at least two of the
singularities are non-removable. Let $p', q'$ be labels of singular points for surfaces in
$\HH'$, corresponding to indices $i \neq j$. To simplify notation
assume $i=1, j=2$. There are three cases to
consider: $a_i = a_j=0$, or one of $a_i, a_j$ are positive, or both
are positive.

If $a_i = a_j=0$ then by assumption $k \geq  4$. We take a cylinder $C$ on a fully
stably completely periodic surface $M$ in $\HH = \HH(a_1, \ldots,
\hat{a}_i, \ldots, \hat{a}_j, \ldots, a_k)$. The notation $\hat{a}_i$
means that the symbol should be ignored; that is on a stratum of
the same genus with $k-2 \geq 2$ singular points obtained by removing two
removable singular points. We place two singular points marked $p',
q'$ in the interior of $C$ at different heights. If $a_i>0$ and $ a_j=0$
is zero we take a fully stably periodic surface $M$ in $\HH(a_1,
\ldots, a_i-1, \ldots, \hat{a}_j, \ldots, a_k)$, find a cylinder $C$
on $M$ whose boundary component is made of saddle connections starting
and ending at $\xi_i$, place a marked point labeled $\xi_j$ in the
interior of $C$. If $a_i$ and $a_j$ are both positive we use the
induction hypothesis to find a surface $M \in \HH(a_1, \ldots, a_i-1,
\ldots, a_j-1, \ldots, a_k)$ with a cylinder whose boundary components
see $\xi_i$ and $\xi_j$, and we perform the surgery in Lemma \ref{lem:
  basic surgery} to 
this cylinder. 
\end{proof}

\begin{lem}[Two surgeries involving genus two surfaces] \label{lem: a pair of surgeries}
  Let $\HH = \mathcal{H}(b_1, \ldots,b_k)$ be a 
  stratum of translation surfaces and let $M \in \HH$ have a
  horizontal cylinder $C$, with circumference 
  $c$. Let $p$ and $q$ be singular points with order $b_i, b_j$
  respectively, such that one boundary component of $C$ only sees
  singularity $p$ and the 
  other only sees singularity $q$. Then for any
  $w_1,w_2>0$ there exists $M' \in 
  \HH' = \mathcal{H}(b_1, \ldots,b_i+2, \ldots,b_j+2, \ldots,b_k)$ which has three
  cylinders $C_1, C_2, C_3$ with circumferences $c+w_1+w_2, w_1$ and
  $w_2$ respectively.  The complements $M \sm C$ and $M' \sm (C_1 \cup
  C_2 \cup C_3)$ are isometric by an isometry preserving the labels of
  singular points,  and $C_1, C_2, C_3$ all have one boundary
  component that sees only $p$, and another that sees only $q$. Thus,
  if $M$ is fully stably 
   periodic so is $M'$. Moreover,  if the $b_i$ are all even, so that
   $\HH' $ has even and odd spin 
  components, we can choose $M'$ to be in either the even or odd
  connected component. 
\end{lem}

\begin{proof}
Once again we encourage the reader to consult Figures \ref{fig: second surgery A} and
\ref{fig: second surgery B}.

\begin{figure}
\begin{tikzpicture}
\draw (0,0)--(6,0)--(6,3)--(1,3)--(1.3,5)--(.3,5)--(0,3)--(0,0);
\draw (6.7,0)--(8.7,0)--(8.7,-1)--(10,-1)--(10,3)--(8.7,3)--(8.7,4.2)--(6.7,4.2)--(6.7,0);
\draw[green] (6,0)--(6,3); 
\draw[green] (6.7,0)--(6.7,3);
\node at (6,3) [circle, draw, fill=white, outer sep=0pt,  inner
sep=1.5pt]{};
\node at (8.7,0) [circle, draw, fill=black, outer sep=0pt,  inner
sep=1.5pt]{};
\node at (8.7,3) [circle, draw, fill=white, outer sep=0pt,  inner
sep=1.5pt]{};
\node at (6,0) [circle, draw, fill=black, outer sep=0pt,  inner
sep=1.5pt]{};
\node at (8.7,4.2) [circle, draw, fill=black, outer sep=0pt,  inner
sep=1.5pt]{};
\node at (6.7,4.2) [circle, draw, fill=black, outer sep=0pt,  inner
sep=1.5pt]{};
\node at (0,3) [circle, draw, fill=white, outer sep=0pt,  inner
sep=1.5pt]{};
\node at (10,3) [circle, draw, fill=white, outer sep=0pt,  inner
sep=1.5pt]{};
\node at (0,0) [circle, draw, fill=black, outer sep=0pt,  inner
sep=1.5pt]{};
\node at (10,0) [circle, draw, fill=black, outer sep=0pt,  inner
sep=1.5pt]{};
\node at (6.7,0) [circle, draw, fill=black, outer sep=0pt,  inner
sep=1.5pt]{};
\node at (6.7,3) [circle, draw, fill=white, outer sep=0pt,  inner
sep=1.5pt]{};
\draw[<->,dashed] (.05,1)--(5.95,1);
\node at (3,1.4) {$c$};
\draw[<->,dotted] (6.75,2)--(8.6,2);
\draw[<->] (8.75,.5)--(9.95,.5);
\node at (9.5,.8) {$w'$};
\node at (7.75,2.3) {$w$};
\node at (7.7,0) {/};
\node at (7.7,4.2) {/};
\node at (9.3,-1){//};
\node at (9.3,3) {//};
\node at (0,1.5) {$\triangle$};
\node at (10,1.5) {$\triangle$};
\node at (8.7,3.6){$\square$};
\node at (6.7,3.6){$\square$};
\node at (8.7,-.5){$\nabla$};
\node at (10,-.5){$\nabla$};
\end{tikzpicture}
\caption{First option for $M'$ in Lemma \ref{lem: a pair of
    surgeries}. Attaching the subsurface on the right increases the
  genus by 2. Unlabeled edges are
  attached to $M' \sm (C_1 \cup C_2 \cup C_3)$. } \label{fig: second surgery A}
\end{figure}
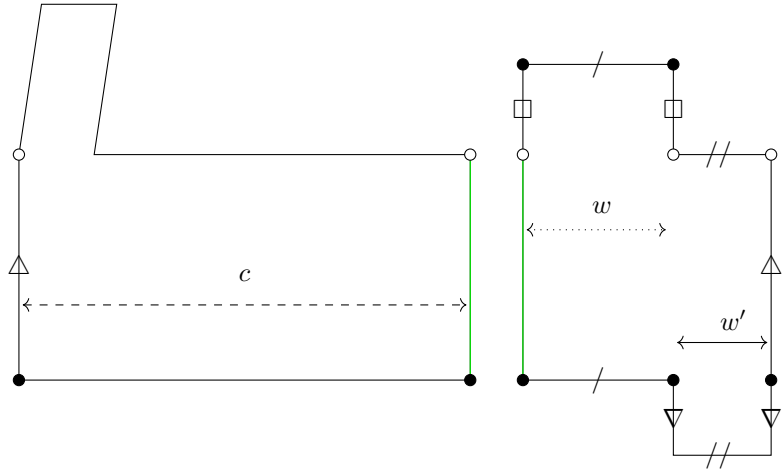

\begin{figure}
\begin{tikzpicture}
\draw (0,0)--(6,0)--(6,3)--(1,3)--(1.3,5)--(.3,5)--(0,3)--(0,0);
\draw (6.7,0)--(8.7,0)--(8.7,-1)--(10,-1)--(10,4.2)--(8,4.2)--(8,3)--(6.7,3);
\draw[green] (6,0)--(6,3);
\draw[green] (6.7,0)--(6.7,3);
\node at (6,3) [circle, draw, fill=white, outer sep=0pt,  inner
sep=1.5pt]{};
\node at (8.7,0) [circle, draw, fill=black, outer sep=0pt,  inner
sep=1.5pt]{};
\node at (6,0) [circle, draw, fill=black, outer sep=0pt,  inner
sep=1.5pt]{};
\node at (10,4.2) [circle, draw, fill=black, outer sep=0pt,  inner
sep=1.5pt]{};
\node at (0,3) [circle, draw, fill=white, outer sep=0pt,  inner
sep=1.5pt]{};
\node at (0,0) [circle, draw, fill=black, outer sep=0pt,  inner
sep=1.5pt]{};
\node at (8.7,-1) [circle, draw, fill=white, outer sep=0pt,  inner
sep=1.5pt]{};
\node at (10,-1) [circle, draw, fill=white, outer sep=0pt,  inner
sep=1.5pt]{};
\node at (8,4.2) [circle, draw, fill=black, outer sep=0pt,  inner
sep=1.5pt]{};
\node at (10,0) [circle, draw, fill=black, outer sep=0pt,  inner
sep=1.5pt]{};
\node at (8,3) [circle, draw, fill=white, outer sep=0pt,  inner
sep=1.5pt]{};
\node at (10,3) [circle, draw, fill=white, outer sep=0pt,  inner
sep=1.5pt]{};
\draw[<->,dashed] (.05,1)--(5.95,1);
\node at (6.7,0) [circle, draw, fill=black, outer sep=0pt,  inner
sep=1.5pt]{};
\node at (6.7,3) [circle, draw, fill=white, outer sep=0pt,  inner
sep=1.5pt]{};
\node at (3,1.25) {$c$};
\draw[<->,dotted] (6.75,2)--(8.65,2);
\draw[<->] (8.75,.5)--(9.95,.5);
\node at (9.4,.8) {$w'$};
\node at (7.7,2.3) {$w$};
\node at (7.7,0) {/};
\node at (9,4.2) {/};
\node at (9.4,-1){//};
\node at (7.5,3) {//};
\node at (0,1.5) {$\triangle$};
\node at (10,1.5) {$\triangle$};
\node at (8.7,-.5){$\nabla$};
\node at (10,-.5){$\nabla$};
\node at (10,3.6){$\square$};
\node at (8,3.6){$\square$};
\end{tikzpicture}
\caption{Second option for $M'$, with a different spin. }
 \label{fig: second surgery B}

\end{figure}
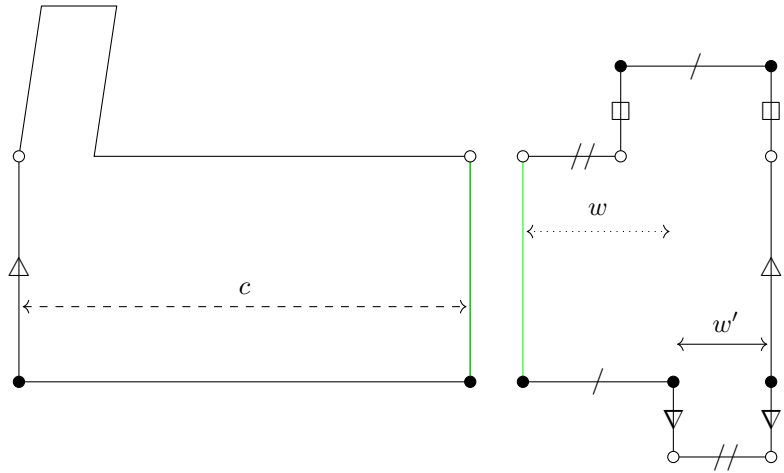

In Lemma \ref{lem: basic surgery} we made a slit in $M$, running through $P$
from top to bottom, and glued in a torus with a slit. In this case we
make an identical slit, this time gluing in a genus two surface with a
slit. This surface is presented in Figures \ref{fig: second surgery A}
and \ref{fig: second surgery B} as made 
up of three rectangles. It is straightforward to check that 
 $M' \in \HH'$ and that it has cylinders satisfying the desired
 properties. It remains to check the final
assertion about the parity of the spin structure.

Recall from \cite[eqn. (4)]{KZ} that where defined, the spin structure of a
surface $M$ of genus $g$ can be computed as
follows. Let $\alpha_i, \beta_j $ (where $1 \leq i,j \leq g$) be a
symplectic basis for $H_1(M)$, realized explicitly as smooth
curves on $M$. This means that all of these curves are disjoint,
except for $\alpha_i$ and $\beta_i$ which intersect once. For each
curve $\gamma$, let $\mathrm{ind}(\gamma)$ be the turning index, that
is the total number of circles made by the tangent vector to $\gamma$,
as one goes around $\gamma$. The parity
of $M$ is then the parity of the integer $\sum_{i=1}^g (1 +\mathrm{ind}(\alpha_i)) (1
+\mathrm{ind}(\beta_i)) $. It is shown in \cite{KZ} that this number
is well-defined (independent of the choice of the symplectic basis)
when all the singular points have even order.

Suppose $M$ has genus $g$ and is equipped with a symplectic
basis. Since any non-separating simple closed curve can be completed
to a symplectic basis, we 
can assume that $\alpha_1$ is the core curve of $C$, and the other
curves in the basis do not intersect the saddle connection from $p$ to $q$ passing
through $C$. We construct a symplectic basis for $M'$ in both cases,
by modifying $\alpha_1$, keeping $\alpha_2, \ldots, \alpha_g, \beta_1,
\ldots, \beta_g$, and adding new curves $\alpha_{g+1},
\alpha_{g+2}, \beta_{g+1}, \beta_{g+2}$. The modified curves are shown
in Figures \ref{fig: symplectic basis1}, \ref{fig: symplectic basis2}, and the
reader can easily check that these 
new curves still form a symplectic basis, and that these two choices
add two  numbers of different parities to the spin structure. 
  \end{proof}

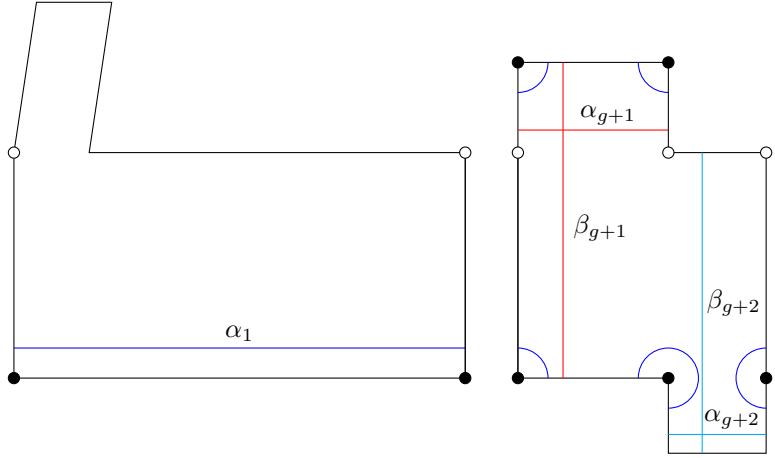
\begin{figure}
\begin{tikzpicture}
\draw (0,0)--(6,0)--(6,3)--(1,3)--(1.3,5)--(.3,5)--(0,3)--(0,0);
\draw (6.7,0)--(8.7,0)--(8.7,-1)--(10,-1)--(10,3)--(8.7,3)--(8.7,4.2)--(6.7,4.2)--(6.7,0);
\draw (6,0)--(6,3); 
\draw (6.7,0)--(6.7,3);
\node at (6,3) [circle, draw, fill=white, outer sep=0pt,  inner
sep=1.5pt]{};
\node at (8.7,0) [circle, draw, fill=black, outer sep=0pt,  inner
sep=1.5pt]{};
\node at (8.7,3) [circle, draw, fill=white, outer sep=0pt,  inner
sep=1.5pt]{};
\node at (6,0) [circle, draw, fill=black, outer sep=0pt,  inner
sep=1.5pt]{};
\node at (8.7,4.2) [circle, draw, fill=black, outer sep=0pt,  inner
sep=1.5pt]{};
\node at (6.7,4.2) [circle, draw, fill=black, outer sep=0pt,  inner
sep=1.5pt]{};
\node at (0,3) [circle, draw, fill=white, outer sep=0pt,  inner
sep=1.5pt]{};
\node at (10,3) [circle, draw, fill=white, outer sep=0pt,  inner
sep=1.5pt]{};
\node at (0,0) [circle, draw, fill=black, outer sep=0pt,  inner
sep=1.5pt]{};
\node at (10,0) [circle, draw, fill=black, outer sep=0pt,  inner
sep=1.5pt]{};
\node at (6.7,0) [circle, draw, fill=black, outer sep=0pt,  inner
sep=1.5pt]{};
\node at (6.7,3) [circle, draw, fill=white, outer sep=0pt,  inner
sep=1.5pt]{};
\draw[blue] (0,0.4)--(6,0.4);
\draw[red] (6.7, 3.3)--(8.7,3.3);
\draw[cyan] (8.7, -.75)--(10,-.75);
\draw[cyan] (9.15, 3)--(9.15,-1);
\draw[red] (7.3, 4.2)--(7.3, 0);
\node at (7.9, 3.5) {$\alpha_{g+1}$};
\node at (7.8, 2) {$\beta_{g+1}$};
\node at (9.57, 1) {$\beta_{g+2}$};
\node at (9.55, -.55) {$\alpha_{g+2}$};
\node at (3,.6) {$\alpha_1$};
\draw[blue] (6.7, 0.4) arc (90:0:0.4);
\draw[blue] (7.1, 4.2) arc (0:-90:0.4);
\draw[blue] (8.7, 3.8) arc (-90:-180:0.4);
\draw[blue] (8.7, -0.4) arc (270:540:0.4);
\draw[blue] (10, -0.4) arc (270:90:0.4);
\end{tikzpicture}
\caption{Modifying the symplectic basis. Gluings as in Figure \ref{fig: second surgery A}.} \label{fig:
  symplectic basis1}
\end{figure}

\begin{figure}
\begin{tikzpicture}
\draw (0,0)--(6,0)--(6,3)--(1,3)--(1.3,5)--(.3,5)--(0,3)--(0,0);
\draw (6.7,0)--(8.7,0)--(8.7,-1)--(10,-1)--(10,4.2)--(8,4.2)--(8,3)--(6.7,3);
\draw (6,0)--(6,3);
\draw (6.7,0)--(6.7,3);
\node at (6,3) [circle, draw, fill=white, outer sep=0pt,  inner
sep=1.5pt]{};
\node at (8.7,0) [circle, draw, fill=black, outer sep=0pt,  inner
sep=1.5pt]{};
\node at (6,0) [circle, draw, fill=black, outer sep=0pt,  inner
sep=1.5pt]{};
\node at (10,4.2) [circle, draw, fill=black, outer sep=0pt,  inner
sep=1.5pt]{};
\node at (0,3) [circle, draw, fill=white, outer sep=0pt,  inner
sep=1.5pt]{};
\node at (0,0) [circle, draw, fill=black, outer sep=0pt,  inner
sep=1.5pt]{};
\node at (8.7,-1) [circle, draw, fill=white, outer sep=0pt,  inner
sep=1.5pt]{};
\node at (10,-1) [circle, draw, fill=white, outer sep=0pt,  inner
sep=1.5pt]{};
\node at (8,4.2) [circle, draw, fill=black, outer sep=0pt,  inner
sep=1.5pt]{};
\node at (10,0) [circle, draw, fill=black, outer sep=0pt,  inner
sep=1.5pt]{};
\node at (8,3) [circle, draw, fill=white, outer sep=0pt,  inner
sep=1.5pt]{};
\node at (10,3) [circle, draw, fill=white, outer sep=0pt,  inner
sep=1.5pt]{};
\node at (6.7,0) [circle, draw, fill=black, outer sep=0pt,  inner
sep=1.5pt]{};
\node at (6.7,3) [circle, draw, fill=white, outer sep=0pt,  inner
sep=1.5pt]{};
\draw[blue] (0,0.4)--(6,0.4);
\draw[blue] (6.7, 0.4) arc (90:0:0.4);
\draw[blue] (8.4, 4.2) arc (0:-90:0.4);
\draw[blue] (10, 3.8) arc (-90:-180:0.4);
\draw[blue] (8.7, -0.4) arc (270:540:0.4);
\draw[blue] (10, -0.4) arc (270:90:0.4);
\node at (3,.6) {$\alpha_1$};
\draw[red] (8, 3.6)--(10,3.6);
\node at (9.3, 3.8) {$\alpha_{g+1}$};
\node at (7.8, 0.7) {$\beta_{g+1}$};
\node at (9.55, 1.55) {$\beta_{g+2}$};
\draw[cyan] (8.7, -.8)--(10,-.8);
\draw [red] plot [smooth, tension=1] coordinates {(8.55, 4.05)
  (8.55,4) (8.5,2.5) (7.5,1.5) (7.3,0)};
\draw [cyan] plot [smooth, tension=1] coordinates {(9.6,-1)
  (9.6,-.4) 
  (9.1,1.5) (9.7, 2.9) (10,3.2)};
\draw [cyan] plot [smooth, tension=1] coordinates {(8,3.2)
  (8.2,3.1) 
  (7.95, 2.75) (7.7,3)};
\node at (9.15, -.6) {$\alpha_{g+2}$};
\end{tikzpicture}
\caption{Modifying the  symplectic basis, second case. Gluings as in
  Figure \ref{fig: second surgery B}. Note the change in the
  rotation number of $\beta_{g+2}$.}
 \label{fig: symplectic basis2}

\end{figure}
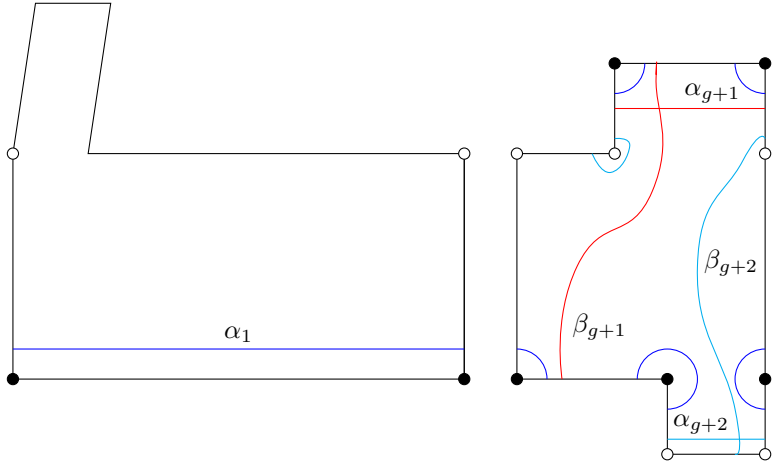

Note that in Proposition \ref{prop:verifying crit} we care about all
connected components of strata. We need to record some
information about the classification of connected components of
strata, due to Kontsevich and Zorich. A
translation surface is {\em hyperelliptic} if it 
admits an involution which acts on absolute homology as
$-\mathrm{Id}$ (see \cite{FM} or \cite[\S 2.1]{KZ} for
more details). A connected
component of a stratum is {\em hyperelliptic} if all surfaces in the
component are hyperelliptic. 

\begin{prop}[\cite{KZ}, Theorems 1 \& 5 and Corollary 5 of Appendix B] \label{prop: KZ}
  Let $\HH(a_1, \ldots, a_k)$ be a stratum with $a_i>0$ for all
  $i$. The following holds: 
  \begin{itemize}

    \item $\HH$ has three connected components in the following
      cases:
      \begin{itemize} \item $k=1, a_1 = 2g-2, g \geq 4.$
        \item $k=2$, $a_1 = a_2 = g-1$, $g \geq 5$ is odd. One is
          hyperelliptic, and the two non-hyperelliptic strata are
          distinguished by the spin invariant. 
              \end{itemize}
 \item
    $\HH$ has two connected components in the following cases:
    \begin{itemize} \item All of the $a_i$ are even, $g \geq
      4$, and either $k \geq 3$ or $a_1 > a_2$. The components are
      distinguished by their spin.  
    \item
      $a_1 = a_2$ and $g$ is either $3$ or is even. One of the
      components is hyperelliptic and the 
      other is not. When $g=3$ the hyperelliptic component is even,
      and the other one is odd. 
    \end{itemize}
      \item
        $\HH$ is connected in all other cases. 
    \end{itemize}

  \end{prop}

\begin{proof}[Proof of Proposition \ref{prop:verifying crit}]
  The proof will be  case-by-case. Here are the cases:

  \begin{enumerate}[(i)]
\item \label{item: first one} $\HH(1,1)$. 
  \item \label{item: fifth one} All the $a_i$ are nonzero and $\HH$ is
    connected.
    \item \label{item: sixth one} All the $a_i$ are nonzero and $\HH$
      has two connected components distinguished by spin.
      \item \label{item: seventh one} All the $a_i$ are nonzero and
        $\HH$ has two connected components distinguished by
        hyperellipticity.
        \item \label{item: eighth one} All the $a_i$ are nonzero and
          $\HH$ has three connected components.
          \item \label{item: ninth one} Some of the $a_i$ are zero.
\end{enumerate}

\textbf{Case \ref{item: first one}.} There is just one connected component
and the desired surface is a $Z$-shaped surface, with three horizontal
cylinders $C_1, C_2, C_3$
of circumferences $c_1, c_1+c_3, c_3$, where $C_1, C_3$ are simple. We
put all of the removable 
singular points in the interior of $C_3$, and choose choose $c_1, c_3$
so that $c_1/(c_1+c_3) \notin \Q$.  It is clear that with these
choices the conditions are satisfied.

\textbf{Case \ref{item: fifth one}.} The stratum $\HH$ is
connected, and we have at least two singularities of positive
order. So with no loss of generality that they are labelled 1 and
2. The result follows from Lemma \ref{lem: basic surgery}, applied to a surface
in $\HH(a_1-1, a_2-1, a_3, \ldots, a_k)$, and taking 
$w\notin c\mathbb{Q}$, so that $w/(c+w) \notin \Q$.

\textbf{Case \ref{item: sixth one}.}
We apply the surgery in Lemma \ref{lem: a pair of surgeries}, with
$w_1/w_2 \notin \Q$. Namely if $p$ and $q$ are labelled $i, j$, we let
$b_i = a_i-2$, $b_j = a_j-2$ and $b_\ell = a_\ell$ for $\ell \neq
i,j$. 

\textbf{Case \ref{item: seventh one}.}
There are two connected 
components. One is hyperelliptic, one is not. This means that $a_1 =
a_2$ and either $g=3$ (in which case $a_1 = a_2 = 2$) or $g \geq 4$ is
even (in which case $a_1 = a_2 = g-1$). In
this case we give explicit surfaces, one in each connected
component. The first surface (the $\HH(2,2)$ case is shown in Figure
\ref{fig: a surface in H22}) is a `staircase' surface made of gluing $2g$ rectangles to
each other. The 
rectangles are labelled $(k, B)$ and $(k, T)$ for $k = 1 , \ldots,
g$. The top (respectively, bottom) of $(k, B)$ is glued to the bottom
(resp., top) of $(k,T)$ for $k=1, \ldots, g$, and the left
(resp., right) 
of $(k, T)$ is glued to the right (resp., left) of $(k+1, B)$ for $k=1,
\ldots, g-1$. The horizontal sides of $(1, B)$ are glued to each other,
as are the horizontal sides of $(g, T)$. This surface is
hyperelliptic since it has a hyperelliptic involution rotating each
rectangle around its midpoint, and this involution swaps the
singularities (see \cite[Remark 3]{KZ}). The second surface is obtained as
follows. We first construct a hyperelliptic surface in $\HH(a_1-2,
a_2-2)$ as in the previous paragraph. Then we perform the surgery
described in Lemma \ref{lem: a pair of surgeries}. The resulting
surface has a horizontal cylinder intersecting three vertical
cylinders, and thus, by \cite[Lemma 2.1]{Lindsey}, is not hyperelliptic.  See
Figure \ref{fig: nonhyp} for an example in 
$\HH(2,2)$.  In both of these constructions there are no restrictions
on the sidelengths of the rectangles, and we can easily arrange that
two of the circumferences are incommensurable. 

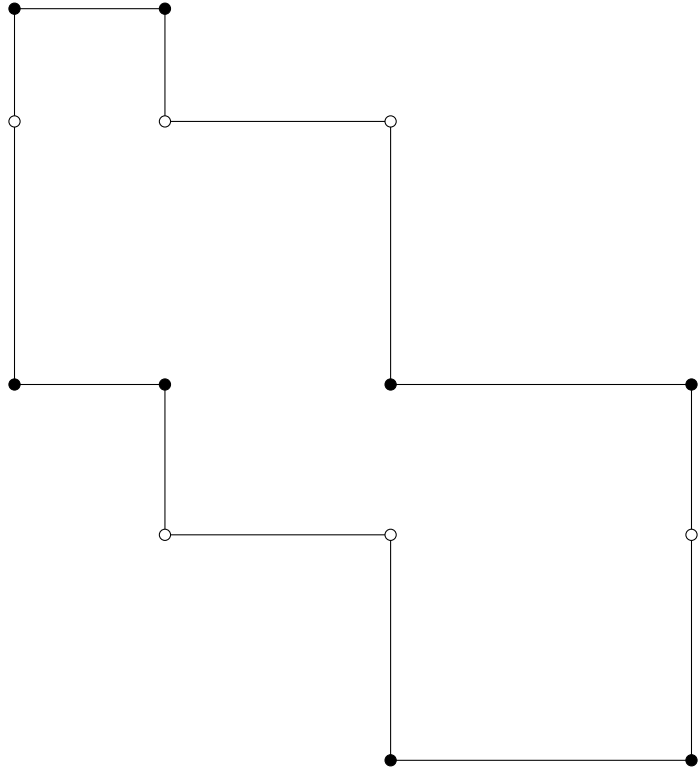
\begin{figure}
\begin{tikzpicture}
\draw(0,0)--(2,0)--(2,-2)--(5,-2)--(5,-5)--(9,-5)--(9,0)--(5,0)--(5,3.5)--(2,3.5)--(2,5)--(0,5)--(0,0);
\node at (0,0) [circle, draw, fill=black, outer sep=0pt,  inner
sep=1.5pt]{};
\node at (2,0) [circle, draw, fill=black, outer sep=0pt,  inner
sep=1.5pt]{};
\node at (5,0) [circle, draw, fill=black, outer sep=0pt,  inner
sep=1.5pt]{};

\node at (9,0) [circle, draw, fill=black, outer sep=0pt,  inner
sep=1.5pt]{};
\node at (9,-5) [circle, draw, fill=black, outer sep=0pt,  inner
sep=1.5pt]{};
\node at (5,-5) [circle, draw, fill=black, outer sep=0pt,  inner
sep=1.5pt]{};

\node at (2,5) [circle, draw, fill=black, outer sep=0pt,  inner
sep=1.5pt]{};
\node at (0,5) [circle, draw, fill=black, outer sep=0pt,  inner
sep=1.5pt]{};

\node at (0,3.5) [circle, draw, fill=white, outer sep=0pt,  inner
sep=1.5pt]{};
\node at (2,3.5) [circle, draw, fill=white, outer sep=0pt,  inner
sep=1.5pt]{};
\node at (5,3.5) [circle, draw, fill=white, outer sep=0pt,  inner
sep=1.5pt]{};

\node at (2,-2) [circle, draw, fill=white, outer sep=0pt,  inner
sep=1.5pt]{};
\node at (5,-2) [circle, draw, fill=white, outer sep=0pt,  inner
sep=1.5pt]{};
\node at (9,-2) [circle, draw, fill=white, outer sep=0pt,  inner
sep=1.5pt]{};
\end{tikzpicture}
\caption{A surface in $\mathcal{H}^{hyp}(2,2)$.}\label{fig: a surface 
    in H22}
\end{figure}

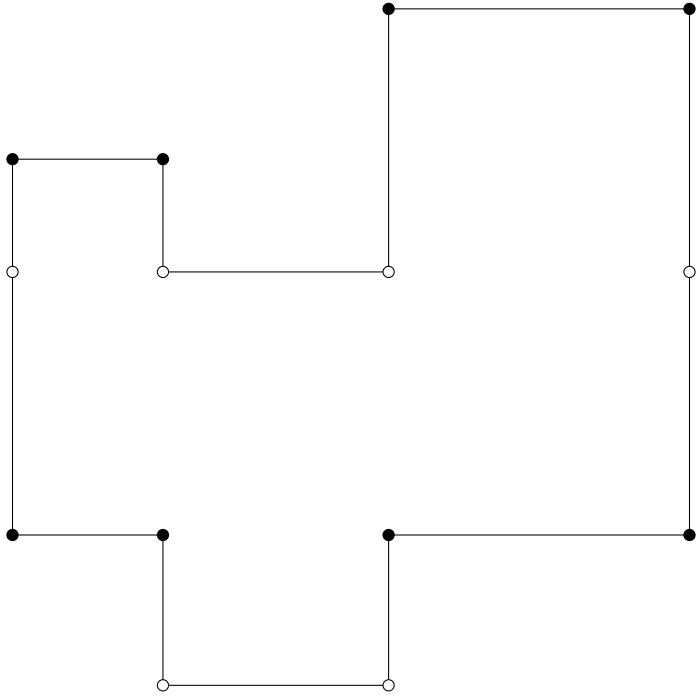
\begin{figure}
\begin{tikzpicture}
\draw (5,3.5)--(2,3.5)--(2,5)--(0,5)--(0,0)--(2,0)--(2,-2)--(5,-2)--(5,0)--(9,0)--(9,7)--(5,7)--(5,3.5);
\node at (0,0) [circle, draw, fill=black, outer sep=0pt,  inner
sep=1.5pt]{};
\node at (2,0) [circle, draw, fill=black, outer sep=0pt,  inner
sep=1.5pt]{};
\node at (5,0) [circle, draw, fill=black, outer sep=0pt,  inner
sep=1.5pt]{};

\node at (9,0) [circle, draw, fill=black, outer sep=0pt,  inner
sep=1.5pt]{};
\node at (9,7) [circle, draw, fill=black, outer sep=0pt,  inner
sep=1.5pt]{};
\node at (5,7) [circle, draw, fill=black, outer sep=0pt,  inner
sep=1.5pt]{};

\node at (2,5) [circle, draw, fill=black, outer sep=0pt,  inner
sep=1.5pt]{};
\node at (0,5) [circle, draw, fill=black, outer sep=0pt,  inner
sep=1.5pt]{};

\node at (0,3.5) [circle, draw, fill=white, outer sep=0pt,  inner
sep=1.5pt]{};
\node at (2,3.5) [circle, draw, fill=white, outer sep=0pt,  inner
sep=1.5pt]{};
\node at (5,3.5) [circle, draw, fill=white, outer sep=0pt,  inner
sep=1.5pt]{};

\node at (2,-2) [circle, draw, fill=white, outer sep=0pt,  inner
sep=1.5pt]{};
\node at (5,-2) [circle, draw, fill=white, outer sep=0pt,  inner
sep=1.5pt]{};
\node at (9,3.5) [circle, draw, fill=white, outer sep=0pt,  inner
sep=1.5pt]{};
\end{tikzpicture}
\caption{A surface in $\mathcal{H}^{nonhyp}(2,2)$.}\label{fig: nonhyp}
\end{figure}

\textbf{Case \ref{item: eighth one}.}
In this case $a_1 = a_2 = g-1$ for $g \geq 5$ odd. Applying the
argument in Case \ref{item: sixth one}, we obtain the required
surfaces in the odd and even connected components. To obtain the required
surface in the hyperelliptic component we use the `staircase surface'
describe in Case \ref{item: seventh one}. 

\textbf{Case \ref{item: ninth one}. } Assume with no loss of
generality that the removable singularities are labelled $k'+1,
\ldots, k$ for some $k' \geq 2$, and let $\HH' = \HH(a_1, \ldots,
a_{k'})$. Note that the singularities $p$ and $q$ have label in $\{1,
\ldots, k'\}$. Apply the preceding considerations to obtain a surface in
$\HH'$ with the required cylinders. By examining the proof in all
preceding case one sees that the number of horizontal cylinders on
this surface is at least three, that is there is at least one cylinder
$C_3$ which is distinct from the cylinders $C_1, C_2$, and we modify
$M$ by adding $k-k'$ in general position in the interior of $C_3$, to
obtain the desired surface.

\end{proof}

\section{Zero entropy} \label{sec: zero entropy}
In this section we prove the following result:
\begin{thm}\label{thm: rel entropy general}
Let $\cH$ be a stratum for which $\dim Z >0$, let $z \in Z \sm \{0\}$, and let $\mu$ be a
probability measure on $\cH$ such that $\Rel_{z} (q)$ is defined for
$\mu$-a.e. $q$. Assume that $\mu$ is $\Rel_{z}$-invariant and ergodic,
and assume in addition that
\begin{equation}\label{eq: moreover}
  \text{ there is  } t_n \to
\infty \text{ so that } (g_{-t_n})_* \mu
\text{ converges to a probability measure}
\end{equation}
(in the weak-* topology). Then the
entropy of $\Rel_z$ acting  on $(\cH, \mu)$ is zero. 
\end{thm}

For the proof of Theorem \ref{thm: rel entropy general}, we will
need an estimate showing that
points stay close to each other for times up 
to $L$, provided their initial distance is polynomially small (as a
function of $L$). To make this precise we will use
the sup-norm Finsler metric on $\cH$, which was introduced  by Avila, Gou\"ezel and Yoccoz
\cite{AGY} and whose definition we now recall. 
For $q_0, q_1$ belonging to the same connected component of a stratum
$\cH$, we write
\begin{equation}\label{eq: inf which is a min}\dist(q_0, q_1)  = \inf_\gamma \int_0^1
  \|\gamma'(t)\|_{\gamma(t)} dt,
  \end{equation}
where $\gamma: [0,1]\to \cH_{\mathrm{m}}$ ranges
over all $C^1$ curves with $\gamma(0) \in \pi^{-1}(q_0), \ \gamma(1)
\in \pi^{-1}(q_1)$, and
$\|\cdot \|_{\mathbf{q}}$ is a pointwise norm on the tangent space to $\cH_{\mathrm{m}}$ at
$\mathbf{q}$, identified via the developing map with $H^1(S, \Sigma;
\R^2)$. 
 Below, balls, diameters of sets, and $\vre$-neighborhoods
 of sets will be defined using this metric. We can now state our
 estimate. 

\begin{prop}\label{prop: local growth}
Let $\cH$ be a stratum of translation surfaces with at least two
singularities, let $Z$ be its real Rel space, let $z_0 \in Z$, and let
$T$ be the map of $\cH$ defined by applying $\Rel_{z_0}$ (where defined). Then for every
compact subset $K 
\subset \cH$, there is $L_0 >0 
$, such that if
$q \in \cH, \ L \in \N, \ L>L_0, 
$ satisfy
\begin{equation}\label{eq: they satisfy}
q \in K \ \ \text{ and } \ g_{-\ell}q \in K, \ \ \text{ where }  \ell \df
2\log L,
\end{equation}
then the maps $T, \ldots, T^L$ are all defined on $B\left(q, \frac{1}{L^{5}}
\right)$, and we have
$$
\max_{j=1, \ldots, L} \diam \left( T^j \left(B
  \left(q, \frac{1
      }{L^{5}} \right)  \right) \right) 
\to_{L \to \infty} 0.
$$

\end{prop}

We have made no attempt to optimize the power 5 in this
statement.

Our proof of Proposition \ref{prop: local growth} will use 
some properties of
the sup-norm metric. They are proved in \cite{AGY}, see also \cite{AG} and
\cite[\S 2]{CSW}. Our notation will follow the one used in \cite{CSW}. 

\begin{prop}\label{prop: sup norm properties}
The following hold:
\begin{enumerate}
  \item[(a)] For all $q_0, q_1$ and all $t \in \R$, 
$
\dist(g_t q_0, g_t q_1) \leq e^{2|t|} \dist(q_0, q_1).
$
\item[(b)]
  The metric dist is proper; that is, for any fixed basepoint $q_0$, the map $q \mapsto \dist(q,
  q_0)$ is proper. In particular, the $\vre$-neighborhood of a
  compact set is pre-compact, for any $\vre>0$. 
\item[(c)]
  The map $\bq \mapsto \| \cdot \|_\bq$ is continuous, and hence bounded
  on compact sets.  This means that for any compact $K \subset \cH_{\mathrm{m}}$
  there is $C>0$ such that for any $\bq_0, \bq_1$ in $K$, the norms $\|
  \cdot \|_{\bq_0} , \, \| \cdot \|_{\bq_1}$ are bi-Lipschitz equivalent
  with constant $C$.
  \item[(d)] The infimum in \eqref{eq: inf which is a min} is actually a
    minimum, that is attained by some curve $\gamma$.
  
\end{enumerate}
\end{prop}

With these preparations, we can give the

\begin{proof}[Proof of Proposition \ref{prop: local growth}]
Write $B \df B\left(q,
  \frac{1}{L^{5}} \right),
\, A \df g_{-\ell}(B)$ (note that $A$ and $B$ both depend on $L$ and
$q$ but
we suppress this from the notation). Let $K'$ be the 1-neighborhood of
$K$, which is
a pre-compact subset of $\cH$ by Proposition \ref{prop: sup norm
  properties}(b). 
Since $\diam(B) \leq \frac{2}{L^{5}}$, Proposition \ref{prop:
  sup norm properties}(a) implies that 
\begin{equation}\label{eq: taking into account}
  \diam(A) \leq  \frac{2}{L^3}.
\end{equation}
It follows from \eqref{eq:
  they satisfy} that  $A \cap K \neq
\varnothing$ and therefore $A \subset K'$.
Since
\begin{equation}
  \label{eq: taking into account 2}
\max_{j=1, \ldots, L}  \|je^{-\ell}
z_0 \| \leq 
\frac{1}{L} \|z_0\| \to_{L \to \infty} 0,
\end{equation}
for all large enough $L$ (depending on $K'$) we have that
$\Rel_{je^{-\ell}z_0}(q')$ is defined for $q' \in K'$. Since $q'_1 \df
\Rel_{je^{-\ell}z_0}\circ g_{-\ell}(q_1)$ is defined for $q_1 \in B$,
we have from \eqref{eq: commutation rel geodesic} that  
$T^j(q_1)  =\Rel_{jz_0}(q_1) = g_{\ell}(q'_1) $ is also defined. This
proves that the maps $T, T^2, \ldots, T^L$ are all defined on $B$.

Furthermore, this computation shows that $T^j(B) = 
g_\ell (\Rel_{je^{-\ell}z_0} (A))$, and so by Proposition \ref{prop:
  sup norm properties}(a), 
it suffices to show that
$$L^2 \cdot \diam \left(\Rel_{je^{-\ell}z_0}
  (A) \right)  \to_{L \to \infty } 0.$$
 Taking into account \eqref{eq: taking into account}
  and \eqref{eq: taking into account 2},  it suffices to show that
 for any
  compact $K'$
  there are positive $\vre, C$ such that for any $q_0, q_1 \in K'$ with $\dist(q_0,
  q_1) < \vre$, and any $z \in Z$ with
  $\|z\|< \vre$, we have
  \begin{equation}\label{eq: locally Lipschitz}
    \dist(\Rel_{z}(q_0, q_1)) \leq C \, \dist(q_0, q_1).
  \end{equation}
  Informally, this is a uniform local Lipschitz estimate for the
  family of maps defined by small elements of $Z$.
  
  To see \eqref{eq: locally Lipschitz}, let $\vre_1$ be small enough so that for any $q \in K'$,
  the ball $B(q, 2\vre_1)$ is contained in a neighborhood which is
  evenly covered by the map $\pi: \cH_{\mathrm{m}} \to \cH$, and let
  $C$ be a bound as in Proposition \ref{prop: sup norm properties}(c),
  corresponding to the compact set which is the $2\vre_1$-neighborhood
  of $K'$. Let  $\vre< \vre_1$ so that for any $z \in Z$ with
  $\|z\|< \vre$ and any $q \in \cH$, $\dist(q, \Rel_{z}(q))<
  \vre_1$. If $\dist(q_0, q_1)<\vre$ then the path $\gamma$ realizing
  their distance (see Proposition \ref{prop: sup norm
    properties}(d)) is contained in a connected component
  $\mathcal{V}$ of
  $\pi^{-1}\left(B(q_0, \vre_1) \right)$. Let
  $$\bar \gamma: [0,1] \to
  H^1(S, \Sigma; \R^2) , \ \ \ \bar \gamma (t)  \df \dev(\gamma(t)) -
  \dev(\gamma(0)),$$  let
  $$\gamma_1 \df \Rel_{z}\circ \gamma, $$
  and analogously define
  $$\bar \gamma_1 : [0,1] \to H^1(S, \Sigma ; \R^2), \ \ \  \bar \gamma_1(t) \df \dev(\gamma_1(t)) -
  \dev(\gamma_1(0)).$$
  By choice of $\vre$ and $\vre_1$, the curve
  $\gamma_1$ also has its image in $\mathcal{V}$. 
 Since $\Rel_{z}$ is
  expressed by $\dev|_{\mathcal{V}}$ as a translation map, the
  curves $\bar \gamma, \bar \gamma_1$ are identical maps. When computing
  $\dist(\Rel_{z}(q_0), \Rel_{z}(q_1))$ via \eqref{eq: inf which
    is a min}, an upper bound is given by computing the path integral
  along the curve $\gamma_1$. We compare  this path integral along
  $\gamma_1$, with the path integral along $\gamma$ giving $\dist(q_0,
  q_1)$. In these two integrals, for any $t$, the tangent
  vectors $\gamma'(t), \gamma'_1(t)$ are identical elements of $H^1(S,
  \Sigma; \R^2))$ for all $t$, 
  but the norms are evaluated using different basepoints. Since these
  basepoints are all in the $2\vre_1$-neighborhood of $K'$, by choice
  of $C$, we have $\|\gamma_1'(t)\|_{\gamma_1(t)} \leq C \|\gamma'(t)
  \|_{\gamma(t)}$ for all $t$. This implies 
  \eqref{eq: locally Lipschitz}. 
\end{proof}

We now list a few additional results we will need. The first is the
following   weak Besicovitch-type covering Lemma, for 
 balls of equal size.

\begin{prop}\label{prop: Besicovitch}
  For any compact $K \subset \cH$ there are positive $N,
      r_0$ so that for
       any $r \in (0, r_0)$, for any $G \subset K$
      the collection $\mathcal{C} \df \{B(q, r): q \in G\}$ contains
      $N$ 
      finite
      subcollections $\mathcal{F}_1, \ldots, \mathcal{F}_N$ satisfying
      $G \subset \bigcup_{i=1}^N \bigcup \mathcal{F}_i, $ and each
      collection $\mathcal{F}_i$ consists of disjoint balls. 
    \end{prop}

    \begin{proof}
The argument is standard, we sketch it for lack of a suitable
reference.
      
      We first claim that given a compact $K$ and $r \in (0, r_0)$,
      there is $N$ so that the largest $r$-separated subset
      of any ball of radius $2r$,  has cardinality at most
      $N$. Indeed, this property is true for Euclidean space by a
      simple volume argument, and is invariant under biLipschitz maps.
      Thus the claim holds by Proposition \ref{prop: sup norm
        properties}(c).

      We now inductively choose the $\mathcal{F}_i$. Let $\mathcal{F}_1$ be a maximal collection of disjoint
      balls of radius $r$ with centers in $G$. For $i \geq 2$, suppose
      $\mathcal{F}_1, \ldots, \mathcal{F}_{i-1}$ have 
      been chosen, let $G_i \df G \sm
      \bigcup_{j=1}^{i-1} \bigcup \mathcal{F}_j$, and let $\mathcal{F}_i$ be the maximal collection
      of disjoint balls of radius $r$ with centers in $G_i$. Clearly $G \supset
      G_1 \supset \cdots \supset G_N$, and we need to show that
      $G_{N+1} = \varnothing$. Since
      $\mathcal{F}_i$ is maximal, for any  $x \in G_i$ there is $x'$
      which is the center of one of the balls of $\mathcal{F}_i$, so
      that $d(x,x') < 2r$.  
      If $x \in G_{N+1} \neq \varnothing$, then the ball
      $B(x, 2r)$ contains $x_1, \ldots, x_N$ such that $x_i$ is
      the center of one of the balls of $\mathcal{F}_i$. For $i'>i$,
      $d(x_i, x_{i'}) \geq r$ since $x_{i'} \in G_{i'}$. This
      contradicts the property of $N$ from the preceding paragraph. 
      \end{proof}

We will  need to know that volumes of
balls do not decay exponentially:
\begin{lem}\label{lem: lemma 1}
  For any probability measure $\mu$ on $\cH$, for
  $\mu$-a.e. $q$, we have
  \begin{equation}\label{eq: do not decay exponentially}
\lim_{r \to 0+} \frac{-\log (\mu(B(q,r)))}{r} =0.
  \end{equation}
  \end{lem}
  \begin{proof}
    If we replace $r$ in the denominator of \eqref{eq:
      do not decay exponentially} with $\log r$, and replace $\lim$
    with $\limsup$, we get the definition
    of the upper pointwise dimension of $\mu$ at $q$. It is known that 
 the upper pointwise Hausdorff dimension of a
measure is at most the Hausdorff dimension of the ambient
space. This
implies the result. 
    \end{proof}

    We also need some standard facts about entropy. In the following 
    proposition, $X$ is a standard Borel space,  $T: X \to X$ is a measurable
        map, $\Prob(X)^T$ denotes the $T$-invariant Borel probability
        measures on $X$,  $\mu$ is a measure in $\Prob(X)^T$, $\mathcal{P}$ is a
        measurable partition of $X$, and $h_\mu(T,
        \mathcal{P})$ is the entropy of $T$ with respect to $\mu$ and
        $\mathcal{P}$. Then the entropy of $T$ with respect to $\mu$
        is $\sup_{\mathcal{P}} h_\mu(T, \mathcal{P})$, where the
        supremum ranges over all finite $\mathcal{P}$.
        For $x \in X$, $\mathcal{P}_n(x)$ is the atom
        of the finite refinement $\bigvee_{i=0}^n T^{-i}\mathcal{P}$
        containing $x$. 

    \begin{prop}\label{prop: entropy standard} We have the following: 
      \begin{enumerate}
      \item
        {\rm [Shannon-McMillan-Breiman Theorem.]}
        If $\mu$ is ergodic then for
        $\mu$-a.e. $x$ we have
        $$\lim_{n \to \infty} \frac{-\log(\mu(\mathcal{P}_n(x)))}{n} =
        h_\mu(T, \mathcal{P}).$$
      \item
        {\rm [Entropy and convex combinations.]}
        If $\mu = \int_{\Prob(X)^T} \nu \ d\theta$, for some probability
        measure $\theta$ on $\Prob(X)^T$, then
        $$
h_{\mu}(T, \mathcal{P}) = \int_{\Prob(X)^T}h_\nu(T, \mathcal{P}) \ d\theta.
$$
\item {\rm [Partitions with small boundary.]} Let $X$ be a locally compact, separable metrizable
  space.  Then
   $h_\mu(T) = \sup_{\mathcal{P} \in \mathrm{Part}_0}
  h_\mu(T, \mathcal{P})$, where $\mathrm{Part}_0$ denotes the
  finite
  partitions  of $X$ into sets $P_i$ satisfy $\mu(\partial P_i)=0$ for
  all $i$. 
        \end{enumerate}
      \end{prop}
 For items (1) and (2) see
    e.g. \cite[Thms. 14.35 \& 15.12]{Glasner} or \cite[Chaps. 2 \&
    3]{ELW}. Item (3) is left as an exercise (see \cite[Pf. of Thm. 2.2]{ELW}).

      \begin{proof}[Proof of Theorem \ref{thm: rel entropy general}]
        We assume that the entropy $h = h_\mu(T)$ satisfies $h>0$,
and we will derive a contradiction. 
  Using Proposition \ref{prop: entropy standard}(3), we choose a partition $\mathcal{P}=\{P_i\}_{i=1}^k$ so that
$\mu(\partial P_i)=0$ for each $i$ and $h_\mu(T, \mathcal{P})>
\frac{h}{2}$.
Choose  $K$ compact so that $\mu(K)>\frac 3 4 $ and 
\begin{equation}\label{eq: exists because}\underset{t\to \infty}{\limsup} \,
  \mu\left(g_{t}(K)\right)>\frac 3 4.
  \end{equation}
A compact set with this property exists by the nondivergence
assumption \eqref{eq: moreover}. Let $N$ and $r_0$ be as in
Proposition \ref{prop: Besicovitch}, for this choice of $K$. 
Using the Shannon-McMillan-Breiman theorem, 
let $L_0$ be large enough so that for all $L > L_0$, the set
$$
W \df \left\{q:  \frac{-\log
      (\mu(P_L(q)))}{L} \leq  \frac{h}{2} \right\}
$$
satisfies
\begin{equation}\label{eq: will be reached}
\mu(W) <   \frac{1}{6N}. 
\end{equation}
Our goal will be choose some $L>L_0$  for which we have  a contradiction to \eqref{eq:
  will be reached}. 

Below we will simplify notation by writing $ B\left(q,
  \frac{1}{L^{5}} \right)$ as $B_{q,L}$ or simply as $B$. 
Let
$$
G_L \df \left\{ q : \mu \left( B_{q, L} \cap W
  \right) >
  \frac{\mu(B_{q, L})}{2} \right\}.
$$
We will show below that
\begin{equation}\label{eq: show a sequence}
  \text{  there are arbitrarily large $L$ for which
  }  \mu(K \cap G_{L}) > \frac{1}{3}.
  \end{equation}
We first explain why \eqref{eq: show a sequence} leads to a
contradiction with \eqref{eq:
  will be reached}. Let $L>L_0$ be large enough so that $\diam(B) \leq
\frac{2}{L^5} < r_0$ and $\mu(K \cap G_L) > \frac{1}{3}$,  
and let
$$
\mathcal{C} \df \left\{ B_{q, L} :  q \in G_{L}
\right\}.
$$
By Proposition \ref{prop: Besicovitch}, there is a subcollection
$\mathcal{F} \subset \mathcal{C}$, consisting of disjoint balls, so
that 
$$
\mu\left(K \cap G_L \cap \bigcup \mathcal{F} \right) \geq \frac{1}{N} \,
\mu(K \cap G_L) >
\frac{1}{3N}. 
$$
Then we have
$$
\mu(W) \geq \mu\left(W 
  \cap \bigcup \mathcal{F}
\right) = \sum_{B \in \mathcal{F}} \mu(W \cap 
B) > \sum_{B \in \mathcal{F}} \frac{\mu(B)}{2}  \geq \frac{\mu
\left(\bigcup \mathcal{F} \right)}{2}\geq \frac{1}{6N},
$$
where the equality follows from the disjointness of
$\mathcal{F}$ and the strict inequality follows from the definitions of
$G_L$ and $\mathcal{C}.$ This gives the desired contradiction to \eqref{eq: will be reached}.

It remains to show \eqref{eq: show a sequence}.
Choose $\vre>0$ so that 
\begin{equation}\label{eq:eps choice}
  21 \, \vre \, \log(k)<
  \frac{h
  }{2}.
\end{equation}
 Given any $L_1$ let
$L>L_1$ and let $X_0 = X_0(L) \subset K$ such that  $\mu(X_0) \geq 
\frac{1}{2},$ and so that for any $q \in X_0$ we have 
 \eqref{eq: they satisfy}. Such $L$ and $X_0$ exist by \eqref{eq:
   exists because}. Using Lemma \ref{lem: lemma 1} we can take  $L$
 large enough so that  
\begin{equation}\label{eq:poly decay} \mu\left(
X_1\right)>\frac  {99}{100}, \ \ \ \text{ where } 
X_1\df     \left\{q
      :\mu\left(B_{q,L} \right ) >k^{-\vre L} 
      \right \} ,
  \end{equation}
  and by making $L$ even larger we can assume that
  \begin{equation}\label{eq: even larger}
k^{-10 \vre L} < \frac{1}{2}. 
    \end{equation}

Now choose $r>0$ so that 
\begin{equation}\label{eq: taking long}
  \mu(V ) < \vre, \ \ \text{ where } \ V \df \left\{y: \dist\left(y,\bigcup_{i=1}^k \partial
      P_i\right)<r\right\} .\end{equation}
This is possible because $\mu \left( \bigcup_i \partial P_i
\right)=0. $

 We claim that 
\begin{equation}\label{eq:not V}
\mu(X_2) > \frac 2 5, \ \ \ \text{ where } X_2 \df 
 \left\{q \in X_0 :  
 |\{0\leq i \leq L: T^iq \in V\}|<10\vre L  \right\}.
\end{equation}
To see this, define
$$E \df \{q:|\{0\leq i \leq L: T^iq \in V\}|\geq 10\vre L\},$$
and let $\mathbf{1}_V$ denote the indicator function of $V$. 
Using \eqref{eq: taking long}, and since $\mu$ is $T$-invariant, 
$$\vre L  >L\mu(V)=\sum_{i=1}^L\int
\mathbf{1}_V\left(T^iq\right) d\mu\geq 10 \vre L \mu(E).$$
Dividing through by $10L \vre$ we have $\mu(E)<\frac{1}{10}$,
giving \eqref{eq:not V}.  

Let 
$$\beta \df k^{-20
\vre L}, \ \ \ \text{ and write } \ \ 
\mathcal{P}^{(L)} \df \bigvee_{i=1}^LT^{-i}(\mathcal{P}).$$
For each
$q$ we let $\mathcal{P}^{(L, B)}$ be the elements of
$\mathcal{P}^{(L)}$ which intersect $B=B_{q,L}$, and partition $\mathcal{P}^{(L,B)}$
into  two subcollections defined
by
$$
\mathcal{P}^{(q, L)}_{\mathrm{big}} \df \left\{P \in\mathcal{P}^{(L,B)}
  : \mu(P)
\geq \beta \mu(B) \right \} \ \ 
\text{ and } \ \ 
\mathcal{P}^{(q, L)}_{\mathrm{small}} \df \mathcal{P}^{(L,B)}
\sm \mathcal{P}^{(q, L)}_{\mathrm{big}}.
$$
We claim that if $q \in X_2$
then
\begin{equation}\label{eq: where beta comes in}
\mu \left(B \cap \bigcup \mathcal{P}^{(q,L)}_{\mathrm{big}}
\right)> \left(1-k^{-10 \vre L} \right) \, \mu \left(B\right) \stackrel{\eqref{eq: even larger}}{>} \frac{\mu(B)}{2}
. 
  \end{equation}
To see this, we 
note that for $q$ satisfying the conclusion of Proposition
\ref{prop: local growth},  the cardinality of $\mathcal{P}^{(L, B)}$ is at most
$k^{|\{0 \leq i \leq L \, : \, T^i q \in V\} |}$.
Indeed, for such $q$, whenever $T^i
q \notin V$, $T^i\left(B\right)$ is contained in one of the
$P_i$ (and for the other $i$ we use the obvious bound that  $T^iq \in
V$, $T^i\left(B \right)$ could  
intersect all of the $P_i$).
For $q \in X_2$ we also have that $\beta^{-1/2} \geq k^{|\{0 \leq i \leq L \, : \, T^i q \in V\} |}$,  and 
this implies that
$$
\mu\left(B \cap \bigcup \mathcal{P}^{(q,L)}_{\mathrm{small}} 
\right) <  \beta^{-1/2} \beta \mu(B) = k^{-10\vre L} \mu(B),
$$
and this proves \eqref{eq: where beta comes in}. 

 If  $q \in X_1 \cap
  X_2$ 
  and $q' \in B_{q, L} \cap \bigcup
\mathcal{P}^{(q,L)}_{\mathrm{big}} $ then we have
$$
\mu(P_L(q')) \geq \beta \mu(B_{q,L})
\geq  k^{-21 \vre L},
$$
and this implies via \eqref{eq:eps choice} that 
$q' \in  W.$ This and \eqref{eq: where beta comes in}  shows that $X_1 \cap X_2 \subset
G_L$.
 Thus
  \begin{equation}\label{eq: 11} \mu(G_L) \geq 
    \mu(X_1 \cap X_2) \geq \frac{2}{5} -
    \frac{1}{100} > \frac 1 3, 
    \end{equation}
and we have shown \eqref{eq: show a sequence}.
\end{proof}

\begin{proof}[Proof of Theorem \ref{thm: rel entropy}]
  Denote by $T$ the map defined by $\Rel_{z_0}$ (where
  defined). Since $m_{\cL}$ is $G$-invariant, it 
is rotation invariant, and thus $m_{\cL}$-a.e. $q$ has no horizontal
saddle connections. In particular for such $q$, $\Rel_z (q)$ is
defined for all $z \in Z$.

Assume first that $m_{\cL}$ is ergodic. By  $G$-invariance of $m_\cL$
we have  \eqref{eq: moreover}, so the 
  hypotheses of 
  Theorem \ref{thm: rel entropy general} are satisfied for $\mu =
  m_{\cL}$. 
  Now suppose $m_{\cL}$ is not ergodic, and let $\mu = \int_{\Prob(X)^T} \nu \ d\theta$ be the ergodic
decomposition of $\mu$, where $\theta$ is a
probability measure on $\Prob(X)^T$ such that $\theta$-a.e. $\nu$ is
ergodic for $T$. By Proposition \ref{prop: entropy standard}(2), it
suffices to show that the entropy of $\nu$ is zero for
$\theta$-a.e. $\nu$, and thus we only need to  show that assumption
\eqref{eq: moreover} holds for $\theta$-a.e. $\nu$. This follows from the
$g_t$-invariance of $m_{\cL}$. Indeed, by invariance and regularity of
$m_{\cL}$, for any $\vre>0$ there exists a compact  $K$, so
that for all $t$,   $m_{\cL}(g_t(K)) =m_{\cL}(K) >1-\vre^2$.
Thus for every $t$, 
$$\theta(\{\nu :g_{-t*}\nu(K) \geq 1-\vre\}) \geq 1-\vre.$$
Thus for any $\vre>0$ there is $K$ so that the set of $\nu$ for which
$(g_{-t_i})_*\nu(K) \geq 1-\vre$ for a sequence $t_i \to \infty$, has
$\theta$-measure at least $1-\vre$.  
Since $\vre$ was arbitrary, we have \eqref{eq: moreover} for $\theta$-a.e. $\nu$. 
\end{proof}

\end{document}